\documentclass[10pt,reqno]{amsart}
\usepackage{amssymb,mathrsfs,graphicx}
\usepackage{ifthen}
\usepackage{colortbl}
\usepackage{cases,appendix}
\definecolor{black}{rgb}{0.0, 0.0, 0.0}
\definecolor{red}{rgb}{1.0, 0.5, 0.5}

\provideboolean{shownotes} 
\setboolean{shownotes}{true} 
\newcommand{\margnote}[1]{
	\ifthenelse{\boolean{shownotes}}%
	{\marginpar{\raggedright\tiny\texttt{#1}}}%
	{}%
}
\newcommand{\hole}[1]{
	\ifthenelse{\boolean{shownotes}}%
	{\begin{center} \fbox{ \rule {.25cm}{0cm} \rule[-.1cm]{0cm}{.4cm}
				\parbox{.85\textwidth}{\begin{center} \texttt{#1}\end{center}} \rule
				{.25cm}{0cm}}\end{center}} {} }

\title[Propagation of chaos for collective behavior models]{Collective behavior models with vision geometrical constraints: truncated noises and propagation of chaos}

\author[Choi]{Young-Pil Choi}
\address[Young-Pil Choi]{\newline Department of Mathematics
	\newline Inha University, Incheon 402--751, Republic of Korea}
\email{ypchoi@inha.ac.kr}

\author[Salem]{Samir Salem}
\address[Samir Salem]{\newline Centre de Math\'ematiques et Informatique (CMI), \newline
	Universit\'e de Provence, Technop\^ole Ch\^ateau-Gombert, Marseille, France}
\email{samir.salem@univ-amu.fr}

\numberwithin{equation}{section}

\newtheorem{theorem}{Theorem}[section]
\newtheorem{lemma}{Lemma}[section]

\newtheorem{proposition}{Proposition}[section]
\newtheorem{remark}{Remark}[section]
\newtheorem{definition}{Definition}[section]

\newcommand{\R}{\mathbb R}

\newcommand{\pp}{\mathcal P}
\newcommand{\e}{\varepsilon}

\newcommand{\lal}{\langle}
\newcommand{\ral}{\rangle}
\newcommand{\mc}{\mathcal{C}}

\newcommand{\lt}{\left}
\newcommand{\rt}{\right}
\newcommand{\pa}{\partial}
 
\newcommand{\mb}{\mathbf{1}}

\newcommand{\bq}{\begin{equation}}
	\newcommand{\eq}{\end{equation}}
\newcommand{\om}{\Omega}

\newcommand{\LL}{\mathcal{L}}

\newcommand{\bv}{\mathbb{B}^d_{V_{m}}}

\newcommand{\mo}{\mathcal{O}}
\newcommand{\E}{\mathbb{E}}

\begin{document}
	\allowdisplaybreaks
	
	\date{\today}
	
	
	
	\begin{abstract} We consider large systems of stochastic interacting particles through discontinuous kernels which has vision geometrical constrains. We rigorously derive a Vlasov-Fokker-Planck type of kinetic mean-field equation from the corresponding stochastic integral inclusion system. More specifically, we construct a global-in-time weak solution to the stochastic integral inclusion system and derive the kinetic equation with the discontinuous kernels and the inhomogeneous noise strength by employing the $1$-Wasserstein distance.
\end{abstract}
	
	\maketitle \centerline{\date}

	\tableofcontents

	%
	%
	%
	%
\section{Introduction}\label{intro}
	
In this paper, 	we are interested in the propagation of chaos for stochastic integral equations (in short, SIEs) describing collective behavior of individuals with vision geometrical constraints. Let $(\Omega,\mathcal{F},(\mathcal{F}_t)_{t\geq 0},\mathbb{P})$ be a probability space endowed with a filtration $(\mathcal{F}_t)_{t \geq 0}$. Here $\Omega$ is the random set, $\mathbb{P}$ and $\mathcal{F}$ are measure and $\sigma$-algebra on that set, respectively. On that probability space, let $\{ B^i_t\}_{i=1}^N$ be $N$ independent $d$-dimensional Brownian motions. In this setting, our main SIEs are given by
\begin{equation}\label{sys_sde_T}
\left\{ \begin{array}{ll}
\displaystyle X_t^i = X_0^i+\int_0^tV_s^i\,ds, \quad i=1,\cdots,N, \quad t \geq 0, & \\[3mm]
\displaystyle V_s^i = V_0^i+\int_0^tF[\mu_s^N](X_s^i,V_s^i)\,ds+\sqrt{2\sigma}\int_0^t\mathcal{R}(V_s^i)\, dB_s^i, \quad \mu_s^N:=\frac{1}{N}\sum_{i=1}^{N}\delta_{(X_s^i,V_s^i)}, 
\end{array} \right.
\end{equation}
where $\mathcal{R} \in \mc^2(\R^d)$ is a truncation function compactly supported in $\bv := \{ x \in \R^d : |x| < V_m\}$. Here $X_t^i$ and $V_t^i$ are position and velocity of $i$-th particle at time $t \geq 0$, respectively, 
and $F[\mu]$ denotes the velocity alignment force given by
\[
F[\mu](x,v) := \int_{\R^d \times \R^d} \mb_{K(v)}(y - x)(w-v)\mu(dy,dw) \quad \mbox{for} \quad \mu \in \pp(\R^d \times \R^d),
\]
where $\mb_{K(v)}$ is the indicator function on the vision set $K(v)$, which is called a communication weight. 

The system \eqref{sys_sde_T} without the noise and vision geometrical constrains is refereed as the Cucker-Smale model \cite{CS}, which is proposed to describe velocity alignment behaviors of individuals, such as flocks of birds or schools of fish, etc. Due to the presence of vision geometrical constraints in \eqref{sys_sde_T}, which comes from $\mb_{K}$ in the interaction force $F$, the individuals are only interacting with others in the velocity dependent region $K$ and trying to align their velocities with others. The system \eqref{sys_sde_T} has a diffusion however it smoothly degenerates when the speed of individuals increases to $V_{m}$. This enables us to consider the situation where the particles moving with higher velocities are less affected by the noises. From this consideration, we can show that the uniform boundedness of speed in time for the system \eqref{sys_sde_T} almost surely under suitable assumption on the initial data, see Lemma \ref{vel_bou_trun}. Thus, the system have a generic property of the velocity alignment models that the speed of individuals cannot be too high.
		
In the original Cucker-Smale model \cite{CS}, a bounded Lipschitz communication weight is considered and obtained the flocking estimate under some condition on the initial configurations. In \cite{MT}, a normalized communication weight is taken into account to deal with interactions between individuals through not only the distance between them but also their relative distance. Recently, in \cite{ACHL, CCH2, CCMP, Pes}, the collision avoidance between individuals is observed by considering a singular communication weight. The influence of noises in Cucker-Smale type models showing flocking or non-flocking behavior is studied in \cite{AH, Choi, DFT,HLL, TLY}. We refer the reader to \cite{CCP, CHL} and the references therein for general survey of flocking models.

Mean field limit and propagation of chaos are some challenging topics in the analysis of partial differential equations(in short, PDEs), which arise in the context of interacting multi-agent systems. Indeed, it provides a rigorous justification of the continuum models. For the original Cucker-Smale model which has regular force fields, the rigorous mean-field limit, existence of weak solutions are studied in \cite{CCR2, CFRT, HL}, see also \cite{CH17, dobru, Mc} for more general types of equations. More recently, Vlasov systems with bounded kernels are taken into account in \cite{JW}, however, their result cannot be directly applied for the system \eqref{sys_sde_T} since the interaction between particles not only depend on their relative distance but also on some topological considerations. For the particle system with singular or non-Lipschitz kernels, the rigorous derivation of continuum descriptions is studied in \cite{BCC,CCH, HJ, Hol, LP}.

For the deterministic case, i.e., the system \eqref{sys_sde_T} without noises, the rigorous derivation of mean-field limit model is studied in \cite{CCHS} in the large particle limit $N \to \infty$. Since the force fields in \eqref{sys_sde_T} are not continuous, the differential inclusion system together with the extended boundary set is introduced, and the quantitative error estimate between solutions to that system and weak solutions to the limiting kinetic equation is obtained. Our main purpose of this paper is to extend the result in \cite{CCHS} to the stochastic case under the same assumption on the sensitivity regions $K$. To be more precise, we will show that the $N$ interacting processes $(X_t^i,V_t^i)$ of the system \eqref{sys_sde_T} respectively well approximates as $N \to \infty$ the processes $(Y_t^i, W_t^i)$ to the following kinetic  McKean-Vlasov type system:
\begin{equation}\label{sys_NLS_T}
\left\{ \begin{array}{ll}
\displaystyle Y_t^i =Y^i_0+\int_0^t W_s^i\,ds, \quad i=1,\cdots,N,~~t > 0, & \\[3mm]
\displaystyle W_t^i =W^i_0+\int_0^t F[f_s](Y_s^i,W_s^i)\,ds + \sqrt{2\sigma}\int_0^t\mathcal{R}(W_s^i)\,dB_s^i, \quad \mathcal{L}(Y_t^i,W_t^i) = f_t, &  \\[3mm]
\left( Y_0^i, W_0^i \right) =\left( X^i_0, V^i_0 \right), \quad i=1,\cdots,N. & 
\end{array} \right.
\end{equation}
Then, by applying It\^o formula, we find that the probability density function $f_t$ is governed by
\bq\label{sys_kin_Trun}
\pa_t f_t + v \cdot \nabla_x f_t + \nabla_v \cdot (F[f_t]f_t) = \sigma\Delta_v\lt(\mathcal{R}^2(v) f_t \rt) , \quad (x,v) \in \R^d \times \R^d,\quad t > 0, 
\eq
with the initial data
\bq\label{ini_sys_kin_Trun}
f_t(x,v)|_{t=0} =: f_0(x,v), \quad (x,v) \in \R^d \times \R^d.
\eq

We emphasize that types of SIEs \eqref{sys_sde_T}, SIEs with discontinuous force fields and truncated diffusion, have not been treated so far to the best of authors' knowledge. It is also worth noticing that since the force fields are discontinuous, the existence of solutions to the SIEs \eqref{sys_sde_T} cannot be obtained by the classical theory of SIEs. Furthermore, it is not possible to use Girsanov's theorem to get a weak notion of solutions as in \cite{HS} due to the consideration of additive noises. Thus our strategy is to replace the SIEs \eqref{sys_sde_T} by a system of stochastic integral inclusion system. For this, we need to use the generalized boundary set $\widetilde{\pa}K$ defined in Definition \ref{def_gb} below as in \cite{CCHS} where the differential inclusion system is introduced for the existence of particle system without noises. To the best of authors' knowledge, there exist no available general integral inclusion theories in the stochastic framework as convenient as Fillipov's theory \cite{Fil} in the deterministic setting. Thus, in the current work, we introduce the corresponding stochastic inclusion system to \eqref{sys_sde_T} and give details of constructing the global-in-time existence of solutions to that system.

{\bf Notations.-} We introduce several notations used throughout the paper. $| \cdot |$ and $\lal \cdot, \cdot \ral$ denote the Euclidean distance and the standard inner product on $\R^d$, respectively. We also use the notation $| \cdot |$ for the Lebesgue measure of some set or the cardinal of finite index sets when there is no confusion. $\pp(\R^d \times \R^d)$ and $\pp_p(\R^d \times \R^d)$ stand for the sets of all probability measures and probability measures with finite moments of order $p \in [1,\infty)$ on $\R^d \times \R^d$, respectively. For a function $f(x,v)$, $\|f\|_{L^p}$ represents the usual $L^p(\R^d \times \R^d)$-norm. For $p \in [1,\infty]$ and $T > 0$, $L^p(0,T; E)$ is the set of the $L^p$ functions from an interval $(0,T)$ to a Banach space $E$. We denote by $C$ a generic positive constant. For a set $A \subset \R^d$, int($A$) and cl($A$) represent the interior and closure of $A$, respectively, and $Tr(M)$ denotes the trace of a matrix $M \in \R^d \times \R^d$. 
	
{\bf Organization of the paper.-} In Section \ref{sec_pre}, we discuss our main mathematical tool, Wasserstein distance. We also present our main assumptions on the sensitivity regions and main results on the existence of solutions to the SIEs and PDEs, and the propagation of chaos. As mentioned before, the deterministic case is already studied in \cite{CCHS}, thus we provide several examples of sensitivity sets $K$ satisfying our main assumptions {\bf (H1)}-{\bf (H2)} below without giving the details of proof. In Section \ref{sec:ext_par}, we present a global-in-time existence of solutions to corresponding stochastic integral inclusion system \eqref{eq:SIIE} to the SIEs \eqref{sys_sde_T}. In Section \ref{sec:ext_sde}, we show the existence and uniqueness of the PDE and its associated nonlinear SIEs. Finally, in Section \ref{sec:mf}, we provide the details of proof for the propagation of chaos for the systems \eqref{sys_sde_T}.
	%
	%
\section{Preliminaries and main results}\label{sec_pre}

\subsection{Wasserstein distance} In this part, we introduce the Wasserstein distance, which is our main mathematical tool to estimate the convergence of the empirical measure for the particle system to the probability measure in law. We also discuss an issue of making use of $1$-Wasserstein distance for our case, and we finally recall from \cite{FG} the estimate of convergence rate of an empirical measure in Wasserstein distance.

For $p \geq 1$ and $\mu, \nu \in \pp_p(\R^n)$, the Wasserstein distance is defined by
\[
\mathcal{W}^p_p(\mu, \nu) := \inf_{\xi \in \Gamma(\mu,\nu)} \int_{\R^n \times \R^n} |x-y|^p \xi(dx,dy)=\inf_{(X,Y)\sim (\mu,\nu)}\E\lt[|X-Y|^p\rt],
\]
where $\Gamma(\mu,\nu)$ is the set of all probability measures on $\R^n \times \R^n$ with first and second marginals $\mu$ and $\nu$, respectively, i.e.,
\[
\int_{\R^n \times \R^n} \phi(x) d\xi(dx,dy) = \int_{\R^n} \phi(x) \mu(dx) \quad \mbox{and} \quad \int_{\R^n \times \R^n} \phi(y) d\xi(dx,dy)= \int_{\R^n} \phi(y) \nu(dy),
\]	
for all $\phi \in \mc_b(\R^n)$, and $(X,Y)$ are all possible couples of random variables with $\mu$ and $\nu$ as respective laws. Note that when $p=1$, the $1$-Wasserstein distance is equivalent to the bounded Lipschitz distance:
\[
\mathcal{W}_1(\mu,\nu) = \sup\lt\{\lt|\int_{\R^n} \varphi(x) \mu(dx) - \int_{\R^n} \varphi(x) \nu(dx) \rt| : \varphi \in \mbox{Lip}(\R^n), \, \mbox{Lip}(\varphi) \leq 1 \rt\},
\]
where Lip($\R^n$) and Lip($\varphi$) represent the set of Lipschitz functions on $\R^n$ and the Lipschitz constant of a function $\varphi$, respectively.

In the current work, it seems convenient to consider $p$-Wasserstein distance with $p \in 2\mathbb{N}$ due to the multiplicative noises. However it has already been pointed out in \cite[Remark 3.1]{CCHS} by the authors and their collaborators that our strategy does not work in Wasserstein distance of order $p$ with $p \in (1,\infty)$, thus the make use of $\mathcal{W}_1$ or $\mathcal{W}_\infty$ is essential in the framework because of the form of force fields, see \cite{CS16}. For those reasons, we introduce a modified Wasserstein 1 distance $\mathcal{W}_1^\gamma$ as
	\[
	\mathcal{W}^{\gamma}_1(\mu,\nu):=\inf_{(X,Y)\sim (\mu,\nu)}\mathbb{E}\lt[\sqrt{\gamma^2+|X-Y|^2}\rt].
	\]
	Note that $\mathcal{W}^\gamma_1$ is not a metric. Employing that quantity enables us to establish stability like estimates for any $\gamma>0$ for the both diffusion and the singularity of the interaction kernel. Finally, by letting $\gamma \to 0$, we provide the results in the desired $1$-Wasserstein metric.
	
Before closing this subsection, we recall from \cite[Theorem 1]{FG} the result on the rate of convergence of the empirical measure in Wasserstein distance, which will be crucially used to obtain our result on propagation of chaos.

\begin{proposition}\label{prop_fg} Let $\mu \in \pp(\R^n)$ and $p \geq 1$. Suppose that $M_q(\mu):= \int_{\R^n} |x|^q \mu(dx) < \infty$ for some $q > p$. Then we have
\[
\mathbb{E}\lt[\mathcal{W}_p^p(\mu^N,\mu)\rt] \leq CM^{p/q}_q(\mu) \left\{ \begin{array}{ll}
N^{-1/2} + N^{-(q-p)/q} & \textrm{if $2p > n$ and $q \neq 2p$}, \\[2mm]
N^{-1/2}\log(1+N) + N^{-(q-p)/q} & \textrm{if $2p = n$ and $q \neq 2p$},\\[2mm]
N^{-p/n} + N^{-(q-p)/q} & \textrm{if $2p < n$ and $q \neq n/(n-p)$},
\end{array} \right.
\]
where $\mu^N = \frac1N\sum_{k=1}^N\delta_{X^k}$ and $C > 0$ depends only on $p,d$ and $q$.

\end{proposition}
	
\subsection{Sharp sensitivity regions} In this subsection, we introduce several notations for the set $K$ and its properties. We also discuss our main assumptions for $K$. 
\begin{definition}Let $K \subset \R^d$ be a non-empty compact set and $\e >0$. We define the $\e$-boundary of $K$ by:
\[
\partial^{\e}K :=\left\{x+y \ | \ x\in \partial K, |y|\leq \e \right\},
\] 
and also the $\e$-enlargement(resp. $\e$-reduction) $K^{\e,+}$ (resp. $K^{\e,-}$) by
\[
K^{\e,+}:=K \cup \partial^{\e}K \quad \mbox{and} \quad K^{\e,-}:=K\setminus \partial^{\e}K
\]
Note that $\pa^\e K = K^{\e,+} \setminus K^{\e,-}$ and $(\pa^\e K)^{\delta,+} \subset \pa^{\e + \delta}K$ for $\e > 0$ and $\delta > 0$.
\end{definition}
We next provide the so called {\it rope argument} used in \cite{Hauray, HS} for the propagation of chaos of Vlasov-Poission or Vlasov-Poisson-Fokker-Planck systems in one dimension whose proof can be found in \cite[Lemma 2.2]{CCHS}.
\begin{lemma}\label{lem_est2}For $K\subset \mathbb{R^d}$ For $x_1,y_1,x_2,y_2 \in \R^d$, we have
\[
|\mb_K(y_1 - x_1) - \mb_K(y_2 - x_2)| \leq \mb_{\pa^{2|x_1 - x_2|}K}(y_1 - x_1) + \mb_{\pa^{2|y_1 - y_2|}K}(y_1 - x_1).
\]
\end{lemma}
We now give our main assumption on the compact set $K(v)$:
	\begin{itemize}
		\item[${\bf (H1)}$] $K(\cdot)$ is globally compact, i.e., $K(v)$ is compact and there exists a compact set $\mathcal{K}$ such that $K(v)\subseteq \mathcal{K}\,, \forall \,v\in \mathbb{R}^d$.
		\item[${\bf (H2)}$]  There exist a family of  closed sets $v \mapsto \Theta(v)$  and a constant $C$ such that:
		\begin{itemize}
			\item[(i)]  $\partial K(v) \subset \Theta(v) $, for all $v \in \R^d$,
			\item[(ii)] $|\Theta(v)^{\e,+}| \le C \e$, for all $\e \in (0,1)$,
			\item[(iii)] $K(v) \Delta K(w) \subset \Theta(v)^{C |v-w|,+}$ for $v,w\in\R^d$,
			\item[(iv)] $\Theta(w) \subset \Theta(v)^{C |v-w|,+}$ for $v,w\in\R^d$,
		\end{itemize}
		where $\Delta$ denotes the symmetric difference, i.e., $A \Delta B = (A\setminus B) \cup (B \setminus A) = (A \cup B)\setminus (A \cap B)$ for $A, B \subset \R^d$.
	\end{itemize}
Before giving some comments on the set-valued function $\Theta(v)$ given in ${\bf (H2)}$, we introduce a generalized boundary of the set $\tilde{\pa}K(v)$ in the definition below.
\begin{definition}\label{def_gb} For $v\in \R^d$, we define the generalized boundary set of $K(v)$, $\widetilde{\pa}K(v)$ by
\[
\widetilde{\pa}K(v) := \lt\{ x \in \R^d : \mathcal{A}(x,v) = [0,1] \rt\},
\]
where $\mathcal{A}: \R^d \times \R^d \to [0,1]$ is given by
\[
\mathcal{A}(x,v) := Conv\lt\{ \alpha \in [0,1] : \exists\, (x^n,v^n) \to (x,v) \mbox{ such that } \mb_{K(v^n)}(x^n) \to \alpha\rt\} \quad \mbox{for} \quad (x,v)\in \R^{2d}.
\]
\end{definition}
Introducing the above boundary set $\widetilde{\pa}K(v)$ is required to give a sense to the time-derivative of the particle trajectories when they cross the boundary of $K(v)$. The set-valued function $\Theta(v)$ is a kind of a regularization of the set valued function $\widetilde{\pa}K(v)$. Note that there are inclusion relations for the sets $\pa K$, $\widetilde\pa K$, and $\Theta$:
\[
\pa K(v)\subset \widetilde{\pa}K(v)\subset \Theta(v) \quad \mbox{for} \quad v \in \R^d.
\]
We refer to \cite{CCHS} for details of its proof. We next provide several examples of sets satisfying the above conditions ${\bf (H1)}$-${\bf (H2)}$ that are studied in \cite[Section 5]{CCHS}. \newline
	
	{\bf (Example 1)} A fixed closed ball in $\R^d$:
	\[
	K(v) = \mbox{cl}(\mathbb{B}_r^d) \quad \mbox{with} \quad r > 0. 
	\]
	
	{\bf (Example 2)} A closed ball with radius evolving regularly with respect to velocity in $\R^d$:
	\[
	K(v) = \mbox{cl}(\mathbb{B}_{r(|v|)}^d) \quad \mbox{with a bounded Lipschitz function } r: \R_+ \to \R_+.
	\]
	
	{\bf (Example 3)} A vision cone in $\R^d$ with $d=2,3$:
	\[
	K(v) = \lt\{ x\,:\, |x| \leq r \quad \mbox{and} \quad -\theta(|v|) \leq \cos^{-1}\lt(\frac{x \cdot v}{|x||v|} \rt) \leq \theta(|v|) \rt\}
	\]
	with $0 < \theta(z) \in \mc^\infty(\R_+)$ satisfying $\theta(z) = \pi$ for $0 \leq z \leq 1$, $\theta(z)$ is decreasing for $z \geq 1$ and $\theta(z) \to \theta_* > 0$ as $|z| \to +\infty$.
	
	For the first two examples, we can choose the generalized boundary set $\Theta(v)$ as $\Theta(v) = \pa K(v)$. For the third one, if we define the family set $\Theta(v)$ as
	\bq\label{def_theta}
	\Theta(v):=\begin{cases}
		\partial C(r,v,\theta(|v|))\cup R(v)& \text{ if } |v|\in(1/2,1), \\ 
		\partial C(r,v,\theta(|v|)) & \text{ else}, 
	\end{cases} 
	\eq   
	where $R(v)=[a(v),b(v)]$ with
	$$
	a(v)=-r\frac{v}{|v|} \quad , \quad b(v)= 2r(|v|-1)\frac{v}{|v|},
	$$
	then the set $\Theta(v)$ defined in \eqref{def_theta} satisfies the condition ${\bf (H2)}$.

\subsection{Main results} In this part, we present our main results of this paper. First, we establish the global-in-time existence of solutions to the stochastic particles system. 
\begin{theorem}\label{thm_SIE} There exists some stochastic basis $(\overline{\om},\overline{\mathcal{F}},(\overline{\mathcal{F}_t})_{t\geq 0},\overline{\mathbb{P}})$, and on this basis a $dN$-dimensional Brownian motion $(B_t^1,\cdots,B_t^N)_{t\geq 0}$ of $2dN$ dimensional $(X_0^1,V_0^1,\cdots,X_0^N,V_0^N)$ random variables with law $f^N_0\in \mathcal{P}_2(\R^{2dN})$ and some $\mathcal{F}_t$-adapted $2dN$ dimensional $(X_t^1,V_t^1,\cdots,X_t^N,V_t^N)_{t\geq0}$ process solution to the following inclusion integral equation
\begin{equation}
\label{eq:SIIE}
\left\{ \begin{array}{ll}
\displaystyle X_t^i = X_0^i+\int_0^tV_s^i\,ds, \quad i=1,\cdots, N, \quad t > 0,&\\[3mm] 
\displaystyle V_s^i= V_0^i+ \frac{1}{N}\sum_{j=1}^N\int_0^t \alpha^{i,j}_{s}
(V_s^j-V_s^i) \,ds+\sqrt{2\sigma}\int_0^t\mathcal{R}(V_s^i)\, dB_s^i, &\\[3mm]
\displaystyle \alpha_{s}^{i,j}\in  \mathcal{I}(X_s^j-X_s^i,V_s^i), \quad \forall \,s\geq 0, &\\[2mm]
\end{array} \right.
\end{equation}
where $\mathcal{I}$ is the set valued function defined as
\[
\mathcal{I}(x,v)=\begin{cases}
\{1\} & \text{ if } x\in \mbox{int}(K(v)) \setminus \widetilde{\pa}K(v),\\[2mm] 
\{0\} & \text{ if } x\in K(v)^c \setminus \widetilde{\pa}K(v), \\[2mm]  
[0,1] & \text{ if } x\in \widetilde{\pa}K(v).
\end{cases}
\]
\end{theorem}
For notational simplicity, we define the set valued function $\tilde{F}[\mu]$ as
\[
\tilde{F}[\mu](x,v):=\int_{\R^d \times \R^d}\mathcal{I}(y-x,v)(w-v)\mu(dydw),
\]
which makes sense at least when $\mu$ is an atomic measure. 

We next state the theorem on the existence of solutions to the nonlinear SIEs \eqref{sys_NLS_T} and its associated PDEs \eqref{sys_kin_Trun}.
	\begin{theorem}\label{thm_SDE} Let $f_0$ be a probability measure on $\R^d \times \R^d$ satisfying $f_0 \in (L^1 \cap L^\infty)(\R^d \times \R^d) \cap \pp_1(\R^d \times \R^d)$ and let $(X_0^i, V_0^i)_{i=1,\cdots, N}$ be $N$ independent variables with law $f_0$. Suppose the initial data $f_0$ is compactly supported in velocity in $\bv$. Then, for some $T > 0$, there exists a unique strong solution $(Y_t^i, W_t^i)_{i=1,\cdots, N}$ to the nonlinear SIEs \eqref{sys_NLS_T} and a unique $f \in L^\infty(0,T; (L^1 \cap L^\infty)(\R^d \times \R^d)) \cap \mc([0,T];\pp_1(\R^d \times \R^d))$ weak solution to \eqref{sys_kin_Trun} which is the law of the process solution to \eqref{sys_NLS_T} and compactly supported in velocity in $\mathbb{B}_{V_m}$ up to time $T >0$. Moreover, if $\tilde{f}\in L^{\infty}(0,T;\mathcal{P}_1(\R^d \times \R^d))$ is another solution starting from $\tilde{f}_0\in \mathcal{P}_1(\R^d \times \R^d)$ then
	\begin{equation*}
	\mathcal{W}_1(f_t,\tilde{f}_t)\leq \mathcal{W}_1(f_0,\tilde{f}_0)e^{\int_0^t \|f_s\|_{L^1\cap L^{\infty}}\, ds}.
	\end{equation*}
	\end{theorem}
	
	\begin{remark}\label{rmk:mom}A straightforward computation yields that for $q \geq 1$
		\[
		\frac{d}{dt}\int_{\R^d \times \R^d} |x|^q f\,dxdv \leq q\int_{\R^d \times \R^d} |x|^{q-1}|v|f\,dxdv \leq C\int_{\R^d \times \R^d} |x|^{q-1} f\,dxdv \leq C\int_{\R^d \times \R^d} |x|^q f\,dxdv + C.
		\]
		Thus the $q$-th moment of $f$ is estimated as
		\[
		\int_{\R^d \times \R^d} |x|^q f_t\,dxdv \leq C\int_{\R^d \times \R^d} |x|^q f_0\,dxdv + C.
		\]
	\end{remark}

Our final result is on the propagation of chaos. For this, we recall the definition of a chaotic sequence and remark the reformulation of the notation of the propagation of chaos in terms of coupling. We refer to \cite{HM} for more details on that.
\begin{definition} Let $f$ be a probability on $\R^{2d}$. A sequence $\lt((X_i^N,V^N_i)_{i \le N} \rt)_{N \in \mathbb{N}}$ of exchangeable random variables is $f$-chaotic if 
\[
\mu^N := \frac1N \sum_{i=1}^N \delta_{(X^N_i,V^N_i)} \overset{\LL}{\longrightarrow} f \quad \mbox{as} \quad N \to \infty.
\]
\end{definition}
\begin{remark}
Assume that $f \in \mathcal{P}_p(\R^{2d})$ endowed with the $\mathcal{W}_p$ metric. Then a sufficient condition for the sequence $\lt((X_i^N,V^N_i)_{i \le N} \rt)_{N \in \mathbb{N}}$ to be $f$-chaotic is 
\[
\E \lt[ \mathcal{W}_p(\mu^N, f)\rt] \to 0 \quad \mbox{as} \quad N \to \infty.
\]
\end{remark}
\begin{theorem}\label{thm:PC_T}
Suppose that the set-valued function $K$ satisfies $\bf{(H1)}$ and $\bf{(H2)}$, and let $f$ be a solution to the system  \eqref{sys_kin_Trun}-\eqref{ini_sys_kin_Trun} up to time $T > 0$, such that $f \in L^\infty(0,T; (L^1 \cap L^\infty)(\R^d \times \R^d)) \cap \mc([0,T]; \pp_1(\R^d \times \R^d))$ with initial data $f_0 \in (L^1 \cap L^\infty)(\R^d \times \R^d) \cap \pp_1(\R^d \times \R^d)$ compactly supported in velocity in $\mathbb{B}_{V_m}$. Let $(X_0^i, V_0^i)_{i=1,\cdots, N}$ be $N$ independent variables with law $f_0$. Furthermore, we assume that $|x|^q f_0 \in L^1(\R^d \times \R^d)$ for $q > 1$. Then there exists a constant $C>0$ depending only on $V_{m}$, $f_0$, $q$ and $T$ such that 
\[
\mathbb{E}\lt[\mathcal{W}_1(\mu_t^N,f_t)\rt] \leq C \left\{ \begin{array}{ll}
N^{-1/2} + N^{-(q-1)/q} & \textrm{if $2 > d$ and $q \neq 2$}, \\[2mm]
N^{-1/2}\log(1+N) + N^{-(q-1)/q} & \textrm{if $2 = d$ and $q \neq 2$},\\[2mm]
N^{-1/d} + N^{-(q-1)/q} & \textrm{if $2 < d$ and $q \neq d/(d-1)$},
\end{array} \right.
\]
for all $t \in [0,T]$, where $\mu_t^N = \frac1N\sum_{i=1}^N \delta_{(X^i_t,V^i_t)}$ is the empirical measure associated to the particle system  \eqref{eq:SIIE}. 
\end{theorem}

	%
	%
	%
	%

\section{Interacting stochastic particle system: Proof of Theorem \ref{thm_SIE}}\label{sec:ext_par}
In this section, we construct a global-in-time solution to the stochastic integral inclusion system \eqref{eq:SIIE} which corresponds to the system \eqref{sys_sde_T} . For this, we regularize the indicator function with respect to phase space $(x,v)$:
\[
\mb^{\eta,\e}_{K(v)}(x) =\mb_K *_{(x,v)} (\phi_\e,\psi_\eta)= \int_{\R^d \times \R^d} \mb_{K(v-w)}(x-y)\phi_\eta(w)\psi_\e(y)\,dydw,
\]
where $\phi_\eta(w) := (1/\eta^d)\phi\left(w/\eta\right)$ with
\[
\phi(v) = \phi(-v) \geq 0, \quad \phi \in \mc^\infty_0(\R^d), \quad \mbox{supp } \phi \subset B(0,1), \quad \mbox{and} \quad \int_{\R^d} \phi(v)\,dv = 1.
\]
Using this newly defined function $\mb^{\eta,\e}_K$, we define $F^{\eta,\e}[\mu]$ as 
\bq\label{eq_f}
F^{\eta,\e}[\mu](x,v)=\int_{\R^d \times \R^d}\mb^{\eta,\e}_{K(v)}(y-x)(w-v)\mu(dy,dw).
\eq
We next recall from \cite[Lemma 4.2]{CCHS} some some basic properties of the regularized indicator function in the lemma below. 
\begin{lemma} (i) For all $ \e > 0$, it holds
\bq\label{lem_diff1}
\int \lt|\mb_{K}^{\e}(x) - \mb_{K}(x)\rt| dx \leq |\pa^{2\e}K|.
\eq
(ii) For all $x \in \mo$ and $0 < \eta \leq 1$, it holds
\bq\label{lem_add}
\int_{\R^d}\lt|\mb_{K(w(x))}^{\eta,\e}(y-x) - \mb_{K(w(x))}^{\e}(y-x)\rt| dy \leq C\eta,
\eq
where $C$ is a positive constant independent of $\e$ and $\eta$.
\end{lemma}
We now consider the following SIEs with smooth force fields and diffusion:
\begin{align}\label{ext_sdeRS}
\begin{aligned}
X_t^{i,\eta,\e}&=X^i_0+ \int_0^t V_s^{i,\eta,\e}\,ds, \quad i = 1,\cdots,N, \quad t > 0,\\
V_t^{i,\eta,\e}&=V^i_0+\frac{1}{N}\sum_{j=1}^N\int_0^t \mb^{\eta,\e}_{K(V_s^{i,\eta,\e})}(X_s^{j,\eta,\e}-X_s^{i,\eta,\e})(V_s^{j,\eta,\e}-V_s^{i,\eta,\e})\,ds+\sqrt{2\sigma}\int_0^t \mathcal{R}(V_s^{i,\eta,\e})\,dB_s^i.
\end{aligned}
\end{align}
Then it is clear that strong existence and uniqueness hold for equation \eqref{ext_sdeRS}. Let us denote by
\[
(\mathcal{X}_t^{N,\eta,\e},\mathcal{V}_t^{N,\eta,\e}):=(X_t^{i,\eta,\e},V_t^{i,\eta,\e})_{i=1,\cdots,N}\quad \mbox{and} \quad \mu_t^{N,\eta,\e}:=\frac{1}{N}\sum_{j=1}^N\delta_{(X_t^{j,\eta,\e}, V_t^{j,\eta,\e})}.
\]
In the lemma below, we estimate the upper bound of the velocity in \eqref{ext_sdeRS} which is useful to control the linear velocity coupling term in the force fields. 
\begin{lemma}\label{vel_bou_trun} Let $(\mathcal{X}_t^{N,\eta,\e},\mathcal{V}_t^{N,\eta,\e})$ be the solution to the system \eqref{ext_sdeRS} on the time interval $[0,T]$. Suppose that 
\[
\max_{1 \leq i \leq N}|V^i_0| \leq V_m,  \quad \mathbb{P} \mbox{- a.s.}
\]
Then it holds 
\[
\max_{1 \leq i \leq N} |V_t^{i,\eta,\e}|\leq V_{m},  \quad \mathbb{P} \mbox{- a.s.},
\]
for $\eta,\e>0$ and $t\in [0,T]$.
\end{lemma}
\begin{proof} We divide the proof into two steps:
\begin{itemize}
\item In {\bf Step A}, we show that the maximal value of $|V^{i,\eta,\e}_t|$ over $i=1,\cdots,N$ has a finite speed of growth in time $t$. 
\item In {\bf Step B}, we show that this maximal value can not exceed the ball of radius $V_{m}$ almost surely, and complete the desired result.
\end{itemize}
$\diamond$ {\bf Step A.-} We set 
\begin{equation*}
A^{\eta,\e}_t:=\max_{1 \leq i \leq N} |V_t^{i,\eta,\e}|.
\end{equation*}
We also notice that if we set 
\[
\Omega_0=\left \{ \omega\in \Omega \ | s\in [0,T]\mapsto\ V_s^{i,\eta,\e}(\omega) \mbox{ is  continuous}  \right\} \in \mathcal{F},
\]
then $\mathbb{P}(\om_0) = 1$ since the paths of the Bwonian motion are almost surely continuous. For $t\in[0,T]$, we define the random (but not stopping) time $\tau_t^{i,\eta,\e}$ as
\begin{equation*}
\tau^{i,\eta,\e}_t=\sup_{s\leq t}\left \{  |V^{i,\eta,\e}_s|=V_{m}\right \}.
\end{equation*}
Then we obtain
\begin{equation*}
\begin{split}
|V^{i,\eta,\e}_t|&= \left |  \int_{\tau^{i,\eta,\e}_t}^t F^{\eta,\e}[\mu_s^{N,\eta,\e}](X_s^{i,\eta,\e},V_s^{i,\eta,\e})\,ds+\sqrt{2\sigma}\int_{\tau^{i,\eta,\e}_t}^t\mathcal{R}(V_s^{i,\eta,\e})\,dB_s +V^{i,\eta,\e}_{\tau^{i,\eta,\e}_t}     \right |\\
&\leq \int_0^t\left | F^{\eta,\e}[\mu_s^{N,\eta,\e}](X_s^{i,\eta,\e},V_s^{i,\eta,\e}) \right |ds+V_{m}+\sqrt{2\sigma}\left |\int_ {\tau^{i,\eta,\e}_t}^t\mathcal{R}(V_s^{i,\eta,\e})\,dB_s \right |
\end{split}
\end{equation*}
Let us define two event sets $\om_A$ and $\om_B$ by	
\[
\Omega_A:=\left \{ [\tau^{i,\eta,\e}_t,t] \subset \left \{ s \in [0,T] \ | \ |V_s^{i,\eta,\e}|\leq V_{m} \right \} \right \},
\]
and
\[
\Omega_B:=\left \{ ]\tau^{i,\eta,\e}_t,t] \subset \left \{ s \in [0,T] \ | \ |V_s^{i,\eta,\e}|> V_{m} \right \} \right \}.
\]
Then by definition of $\tau^{i,\eta,\e}_t$ and the fact that $s \mapsto V_s^{i,\eta,\e}$ is almost surely continuous, we get $\mathbb{P}(\Omega_A)+\mathbb{P}(\Omega_B)=1$. For the event $\om_A$, it is clear to get 
		\[
		|V_s^{i,\eta,\e}|\leq V_{m}.
		\]
		For the event $\Omega_B$,
		\[
		\int_{\tau^{i,\eta,\e}_t}^t\mathcal{R}_{\varepsilon}(V_s^{i,\eta,\e})dB_s=0,
		\]
		and this yields
		\[
		|V_i^{i,\eta,\e}|\leq \int_0^t\left | F^{\eta,\e}[\mu_s^{	N,\eta,\e}](X_s^{i,\eta,\e},V_s^{i,\eta,\e}) \right |ds+ V_{m}.
		\]
		On the other hand, the alignment force term in the above inequality is estimated as
		\[
		\left | F^{\eta,\e}[\mu_s^{N,\eta,\e}](X_s^{i,\eta,\e},V_s^{i,\eta,\e}) \right | \leq \int_{\mathbb{R}^d\times \mathbb{B}_{A^{\eta,\e}_s}^d}\lt(|v|+|V_s^{i,\eta,\e}|\rt)\mu^{N,\eta,\e}_s(dx,dv)\leq A^{\eta,\e}_s+|V_s^{i,\eta,\e}|.
		\]
		This implies 
		\[
		|V_t^{i,\eta,\e}|\leq V_{m}e^t+\int_{0}^t A^{\eta,\e}_s e^{t-s}\,ds, \quad \mathbb{P} \mbox{-a.s.,}
		\]
		and subsequently, by definition of $A^{\eta,\e}_t$, we obtain		
		\[
		A^{\eta,\e}_t\leq V_{m}e^t+\int_{0}^t A^{\eta,\e}_s e^{t-s}ds,
		\]
		and by applying Gronwall's inequality we finally have 
		\begin{equation*}
			A^{\eta,\e}_t\leq V_{m}e^t(1 + t)\leq V_{m}e^T(1 + T), \quad \mathbb{P} \mbox{-a.s.}
		\end{equation*}
$\diamond$ {\bf Step B.-} It follows from the assumption that $A_0\leq V_{m}$. Suppose that for some $\omega\in\om$ it holds $\widetilde{A}(\omega):=\sup_{t\in [0,T]}A_t(\omega)> V_{m}$. Then, for some $i=1,\cdots,N$, there exists $t_0\in [0,T]$ such that
\[
\widetilde{A}(\omega)=|V_{t_0}^{i,\eta,\e}(\omega)|>V_{m}.
\]
We now choose a neighborhood $V_\omega$ of $t_0$ such that $|V_s^{i,\eta,\e}(\omega)| > V_{m}$ for all $s \in V_\omega$. Then $V_t^{i,\eta,\e}(\omega)$ is differentiable in that neighborhood and 
\begin{equation*}
\begin{split}
\frac12\frac{d|V_s^{i,\eta,\e}(\omega)|^2}{ds}&= \int_{\R^d \times \R^d}  \mb_{K(V_s^{i,\eta,\e}(\omega))}(x-X_s^{i,\eta,\e}(\omega))(v-V_s^{i,\eta,\e}(\omega))\cdot V_s^{i,\eta,\e}(\omega)\mu^{N,\eta,\e}_s(\omega)(dx,dv)  \\
				& \leq \left |   \int_{\R^d \times \lt\{|v|\leq \widetilde{A}(\omega)\rt\}}  \underbrace{(v-V_s^{i,\eta,\e}(\omega))\cdot V_s^{i,\eta,\e}(\omega)}_{\leq (\widetilde{A}(\omega)-|V_s^{i,\eta,\e}(\omega)|)|V_s^{i,\eta,\e}(\omega)|}\mu^{N,\eta,\e}_s(\omega)(dx,dv)  \right |\\
				& \leq (\widetilde{A}(\omega)-|V_s^{i,\eta,\e}(\omega)|)|V_s^{i,\eta,\e}(\omega)|.
			\end{split}
		\end{equation*}
		Set
		\bq\label{eq_zs}
		Z^{i,\eta,\e}_s(\omega):=\frac{(\widetilde{A}(\omega)-|V_s^{i,\eta,\e}(\omega)|)^2}{2}.
		\eq
		Then it is straightforward to get 		
		\[
		Z^{i,\eta,\e}_{t_0}(\omega)=0 \quad \mbox{and} \quad 		\left | \frac{dZ^{i,\eta,\e}_s(\omega)}{ds}  \right |=\left | -\frac{d|V^{i,\eta,\e}_s(\omega)|}{ds}\lt(\widetilde{A}(\omega)-|V^{i,\eta,\e}_s(\omega)|\rt) \right |\leq 2Z^{i,\eta,\e}_s(\omega),
		\]
This yields $Z^{i,\eta,\e}_t(\omega)=0$ for $t \in V_\omega$, and in particular $V_{\omega}$ does not depend on $\omega$ and thus $V_{\omega}=[0,T]$. Thus $V^{i,\eta,\e}_t(\omega)> V_{m}$ for all $t\in[0,T]$ due to \eqref{eq_zs}. This subsequently implies the event 
\bq\label{contra}
\lt\{\sup_{t\in [0,T]} \sup_{1 \leq i \leq N} |V^{i,\eta,\e}_t|>V_{m} \rt\},
\eq
is of probability one. On the other hand, by the assumption, the event 
\[
\lt\{\sup_{t\in [0,T]} \sup_{1 \leq i \leq N} |V^i_0|\leq V_{m} \rt\},
\]
is of probability one, which contradicts to \eqref{contra}. This completes the proof.
\end{proof}
We are now in position to give the proof of Theorem \ref{thm_SIE}. As mentioned before, we are going to take into account the generalized boundary set $\widetilde\pa K$ defined in Definition \ref{def_gb} to construct a global-in-time weak solution to the stochastic integral inclusion system \eqref{eq:SIIE}. We remind the reader that the similar strategy is used for the system \eqref{sys_sde_T} without noises, i.e., $\sigma = 0$ in \cite{CCHS}.
\begin{proof}[Proof of Theorem \ref{thm_SIE}]
We divide the proof into two steps.
	
$\diamond$ {\bf Step A: Tightness.-} It follows from Lemma \ref{vel_bou_trun} that if the $(V_i^0)_{i=1,\cdots,N}$ are distributed with a law compactly supported in $\bv$, then $(V_t^{i,\eta,\e})_{i=1,\cdots,N}$ lie inside $\bv$ $\mathbb{P}$-a.s. Using this fact, we find that for $0 \leq s,t \leq T$
$$\begin{aligned}
|\mathcal{X}_t^{N,\eta,\e}-\mathcal{X}_s^{N,\eta,\e}|&\leq CV_m|t-s|\\
|\mathcal{V}_t^{N,\eta,\e}-\mathcal{V}_s^{N,\eta,\e}|&\leq CV_m|t-s|+\sqrt{2\sigma}\sup_{1 \leq i\leq N}\sup_{0 \leq s<r<t \leq T}\lt|\int_s^r \mathcal{R}(V_u^{i,\eta,\e})\,dB_u^i  \rt|,
\end{aligned}$$
for some positive constant $C$. Thus we obtain
\[
|\mathcal{X}_t^{N,\eta,\e}-\mathcal{X}_s^{N,\eta,\e}|+	|\mathcal{V}_t^{N,\eta,\e}-\mathcal{V}_s^{N,\eta,\e}|\leq (C+U_T^{\eta,\e})|t-s|^{1/3},
\]
where
\[
U_T^{\eta,\e}:= \sqrt{2\sigma}\sup_{1 \leq i \leq N}\sup_{0\leq s<r<t\leq T}\lt|\int_s^r \mathcal{R}(V_u^{i,\eta,\e})\,dB_u^i  \rt||t-s|^{-1/3}.
\]
On the other hand, by using Burkholder-Davis-Gundy inequality, we get that for any $p > 1$
$$\begin{aligned}
\E\lt[ \lt(\sup_{1 \leq i \leq N}\sup_{s<r<t}\lt|\int_s^r \mathcal{R}(V_u^{i,\eta,\e})\,dB_u^i  \rt|\rt)^{2p}\rt] &\leq C\sum_{i=1}^N\E\lt[ \sup_{s<r<t}\lt|\int_s^r \mathcal{R}(V_u^{i,\eta,\e})\,dB_u^i  \rt|^{2p}\rt]\\
&\leq C\sum_{i=1}^N C_p\E\lt[ \lt|\int_s^t \mathcal{R}^2(V_s^{i,\eta,\e})\,ds \rt|^p   \rt]\\
&\leq C_{N,p}\|\mathcal{R}\|^2_{L^\infty}|t-s|^p.
\end{aligned}$$
This together with Kolmogorov's continuity theorem gives that for $p>1$ the process 
\[
t\mapsto \sup_{1 \leq i \leq N}\int_s^t \mathcal{R}(V_u^{i,\eta,\e})\,dB_u^i \quad \mbox{is $\gamma$-H\"older $\mathbb{P}$-a.s. for any } \gamma \in \lt(0,\frac{p-1}{2p}\rt).
\]
This subsequently implies $U_T^{\eta,\e}<\infty$ $\mathbb{P}$-almost surely. Let us now denote by
\[
K(R,a):=\lt\{f\in \mc^{1/3}([0,T],\R^{2dN}) : \sup_{0\leq s<t\leq T}\frac{|f(t)-f(s)|}{|t-s|^{1/3}}\leq R \quad \mbox{and} \quad f(0)\in \mathbb{B}^{2dN}_a
\rt\},
\]
which is a compact subset of $C([0,T],\R^{2dN})$ due to Arzel\'a-Ascoli theorem. Then, for all $\eta,\e>0$, we obtain that if $(\mathcal{X}^{N,\eta,\e}_t,\mathcal{V}^{N,\eta,\e}_t)_{t\in [0,T]} \notin K(R,a)$, then either $C+U_T^{\eta,\e}\geq R$ or $|(\mathcal{X}^{N}_0,\mathcal{V}^{N}_0)|\geq a$.
This yields
\[
\mathbb{P}\lt((\mathcal{X}^{N,\eta,\e}_t,\mathcal{V}^{N,\eta,\e}_t)_{t\in [0,T]} \notin K(R,a) \rt)\leq \mathbb{P}\lt(C+U_T^{\eta,\e}\geq R\rt)+\mathbb{P}\lt(|(\mathcal{X}^{N}_0,\mathcal{V}^{N}_0)|\geq a \rt).
\]
On the other hand, since $U_T^{\eta,\e}$ is almost surely finite, we can find some $R > 0$ such that $\mathbb{P}\lt( C+U_T^{\eta,\e}\geq R \rt)\leq a^{-1}$. Moreover, it follows from Chebyshev's inequality that
\[
\mathbb{P}\bigl(|(\mathcal{X}^{N}_0,\mathcal{V}^{N}_0)|\geq a \bigr)\leq \frac{m_2(f_0^N)}{a^2},
\]
where $m_2(f_0^N)$ denotes the second-order moment of $f_0^N$. Hence we have
\[
\sup_{\eta,\e>0}\mathbb{P}\lt( (\mathcal{X}_t^{N,\eta,\e},\mathcal{V}_t^{N,\eta,\e}) \notin K(R,a) \rt)\leq Ca^{-1}(1 \vee a^{-1}),
\]
and this concludes that the family of law $(P^{\eta,\e})_{\eta,\e>0}$ under $\mathbb{P}$ of $(\mathcal{X}^{N,\eta,\e}_t,\mathcal{V}^{N,\eta,\e}_t,B_t)_{t\in [0,T]}$ is tight.

$\diamond$ {\bf Step B: Identification of the limit.-} By {\bf Step A} and Prokhorov's theorem, we can choose a subsequence $(P^{\eta(n),\e(n)})_{n}$ converging to some $P$. Then again by Skorokhod's theorem, we can find a probability space $(\overline{\Omega},\overline{\mathcal{F}},\overline{\mathbb{P}})$ together with some sequence of process $(\overline{\mathcal{X}}^{N,\eta(n),\e(n)}_t,\overline{\mathcal{V}}^{N,\eta(n),\e(n)}_t,\overline{B}_t)_{n}$ and a process $(\overline{\mathcal{X}}^{N}_t,\overline{\mathcal{V}}^{N}_t,W_t)_{t\in [0,T]}$ such that theirs law under $\overline{\mathbb{P}}$ are respectively $(P^{\eta(n),\e(n)})_{n}$ and $P$, and $(\overline{\mathcal{X}}^{N,\eta(n),\e(n)}_.,\overline{\mathcal{V}}^{N,\eta(n),\e(n)}_.,\overline{B}_.)_{n}$ goes $\overline{\mathbb{P}}$-almost surely to $(\overline{\mathcal{X}}^{N}_t,\overline{\mathcal{V}}^{N}_t,W_t)_{t\in [0,T]}$. Note that the third marginal of $P^{\eta(n),\e(n)}$ is the law of a ($dN$)-Brownian motion so that $\overline{B}_.$ is a Brownian motion under $\overline{\mathbb{P}}$. Thus  $((\overline{\Omega},\overline{\mathcal{F}},\overline{\mathbb{P}}), (\overline{\mathcal{X}}^{N,\eta(n),\e(n)}_.,\overline{\mathcal{V}}^{N,\eta(n),\e(n)}_.,\overline{B}_.))$ is a weak solution to \eqref{ext_sdeRS} for each $n$ since uniqueness in law holds for this equation. We now show  that $((\overline{\Omega},\overline{\mathcal{F}},\overline{\mathbb{P}}), (\overline{\mathcal{X}}^{N}_t,\overline{\mathcal{V}}^{N}_t,W_t)_{t\in [0,T]})$ is a weak solution to stochastic integral inclusion system \eqref{eq:SIIE}. Note that since $\overline{\mathcal{X}}^{N,\eta(n),\e(n)}_.,\overline{\mathcal{V}}^{N,\eta(n),\e(n)}_.,\overline{B}_.$ is a weak solution to \eqref{ext_sdeRS}, we obtain that $\overline{\mathbb{P}}$-almost surely 
\begin{align}\label{est_conv}
\begin{aligned}
	\overline{V}_t^i&-\overline{V}_0^i-\frac{1}{N}\sum_{j=1}^N\int_0^t \mb_{K(\overline{V}_s^i)\setminus \tilde{\pa}K(\overline{V}_s^i)}(\overline{X}_s^j-\overline{X}_s^i)(\overline{V}_s^j-\overline{V}_s^i)\,ds-\sqrt{2\sigma}\int_0^t \mathcal{R}(\overline{V}_s^i)\,d\overline{B}_s^i\\
	& = \overline{V}_t^i-\overline{V}_t^{i,\eta,\e}-\sqrt{2\sigma}\int_0^t \bigl(\mathcal{R}(\overline{V}_s^i)-\mathcal{R}(\overline{V}_s^{i,\eta,\e})\bigr)d\overline{B}_s^i - \frac{1}{N}\sum_{j=1}^N\int_0^t I^{\eta,\e}_s\,ds\\
	&=\overline{V}_t^i-\overline{V}_t^{i,\eta,\e}-\sqrt{2\sigma}\int_0^t \bigl(\mathcal{R}(\overline{V}_s^i)-\mathcal{R}(\overline{V}_s^{i,\eta,\e})\bigr)d\overline{B}_s^i-\frac{1}{N}\sum_{j=1}^N\int_0^t I_s^{\eta,\e}\mb_{\{\overline{X}_s^j-\overline{X}_s^i \notin\tilde{\pa}K(V_s^i)) \}}\,ds\cr
	&\quad -\frac{1}{N}\sum_{j=1}^N\int_0^tI_s^{\eta,\e}\mb_{\{\overline{X}_s^j-\overline{X}_s^i \in\tilde{\pa}K(V_s^i)) \}}\,ds,
\end{aligned}
\end{align}
where 
\[
I_s^{\eta,\e} := \mb_{K(\overline{V}_s^i)\setminus \tilde{\pa}K(\overline{V}_s^i)}(\overline{X}_s^j-\overline{X}_s^i)(\overline{V}_s^j-\overline{V}_s^i)-\mb^{\eta,\e}_{K(\overline{V}_s^{i,\eta,\e})}(\overline{X}_s^{j,\eta,\e}-\overline{X}_s^{i,\eta,\e})(\overline{V}_s^{j,\eta,\e}-\overline{V}_s^{i,\eta,\e}).
\]
We then estimate the third and fourth terms on the right hand side of the equality \eqref{est_conv}. By definition of the generalized boundary set $\tilde{\pa}K(\cdot)$, it is clear that if
\[
\overline{X}_s^j- \overline{X}_s^i\notin \tilde{\pa}K(\overline{V}_s^i),
\]
then
\[
\mb^{\eta,\e}_{K(\overline{V}_s^{i,\eta,\e})}(\overline{X}_s^{i,\eta,\e}-\overline{X}_s^{j,\eta,\e})(\overline{V}_s^{j,\eta,\e}-\overline{V}_s^{i,\eta,\e}) \to   \mb_{K(\overline{V}_s^i)\setminus \tilde{\pa}K(\overline{V}_s^i)}(\overline{X}_s^i-\overline{X}_s^j)(\overline{V}_s^j-\overline{V}_s^i),
\]
weakly as $\e,\eta \to 0$. That is,
\[
\frac{1}{N}\sum_{j=1}^N\int_0^t I_s^{\eta,\e}\mb_{\{\overline{X}_s^j-\overline{X}_s^i \notin\tilde{\pa}K(V_s^i)) \}}\,ds \to 0 \quad \mbox{weakly as } \e,\eta \to 0.
\]
We also find 
$$\begin{aligned}
&-\frac{1}{N}\sum_{j=1}^N\int_0^tI_s^{\eta,\e}\mb_{\{\overline{X}_s^j-\overline{X}_s^i \in\tilde{\pa}K(V_s^i)) \}}\,ds\cr
&\quad = \frac{1}{N}\sum_{j=1}^N\int_0^t \bigl(\mb^{\eta,\e}_{K(\overline{V}_s^{i,\eta,\e})}(\overline{X}_s^{j,\eta,\e}-\overline{X}_s^{i,\eta,\e})(\overline{V}_s^{j,\eta,\e}-\overline{V}_s^{i,\eta,\e})\bigr)\mb_{\{\overline{X}_s^j-\overline{X}_s^i \in \tilde{\pa}K(\overline{V}_s^i)\}}\,ds\cr
&\quad = \frac{1}{N}\sum_{j=1}^N\int_0^t \bigl(\mb^{\eta,\e}_{K(\overline{V}_s^{i,\eta,\e})}(\overline{X}_s^{j,\eta,\e}-\overline{X}_s^{i,\eta,\e})(\overline{V}_s^{j,\eta,\e}-\overline{V}_s^{j}+\overline{V}_s^{i}-\overline{V}_s^{i,\eta,\e})\bigr)\mb_{\{\overline{X}_s^j-\overline{X}_s^i \in \tilde{\pa}K(\overline{V}_s^i)\}}\,ds\\
		&\qquad + \frac{1}{N}\sum_{j=1}^N\int_0^t \bigl(\mb^{\eta,\e}_{K(\overline{V}_s^{i,\eta,\e})}(\overline{X}_s^{j,\eta,\e}-\overline{X}_s^{i,\eta,\e})\bigr)(\overline{V}_s^j-\overline{V}_s^i)\mb_{\{\overline{X}_s^j-\overline{X}_s^i \in \tilde{\pa}K(\overline{V}_s^i)\}}\,ds\\
		&\quad =: \int_0^t J_s^{\eta,\e}ds+\frac{1}{N}\sum_{j=1}^N\int_0^t \alpha_{s}^{i,j,\eta,\e}(\overline{V}_s^j-\overline{V}_s^i)\mb_{\{\overline{X}_s^j-\overline{X}_s^i \in \tilde{\pa}K(\overline{V}_s^i)\}}\,ds,
	\end{aligned}$$
where $J^{\eta,\e}_t \to 0$ weakly as $\eta,\e \to 0$ for $t \in [0,T]$ due to the tightness. Thus, by combining the all of the above observations, we have that for any $h > 0$, there exist $\eta,\e > 0$ small enough such that 
\bq\label{est_cl}
\lt|\overline{V}_t^i-\overline{V}_0^i-\frac{1}{N}\sum_{j=1}^N\int_0^t \tilde\mb^{i,j,\eta,\e}_s(\overline{V}_s^j-\overline{V}_s^i)\,ds-\sqrt{2\sigma}\int_0^t \mathcal{R}(\overline{V}_s^i)\,d\overline{B}_s^i\rt|\leq h,
\eq
where
\[
\tilde\mb^{i,j,\eta,\e}_s := \mb_{K(\overline{V}_s^i)\setminus \tilde{\pa}K(\overline{V}_s^i)}(\overline{X}_s^j-\overline{X}_s^i)+\alpha_s^{i,j,\eta,\e}\mb_{\{\overline{X}_s^j-\overline{X}_s^i \in \tilde{\pa}K(\overline{V}_s^i)\}}.
\]
On the other hand, we find that for all $\eta,\e > 0$
\[
\frac{1}{N}\sum_{j=1}^N\int_0^t \tilde\mb^{i,j,\eta,\e}_s(\overline{V}_s^j-\overline{V}_s^i)\,ds\in \int_0^t\tilde{F}[\overline{\mu}_s^N](\overline{X}_s^i,\overline{V}_s^i) \,ds.
\]
Note that the above set is closed. Thus, this together with \eqref{est_cl} yields
\[
\overline{V}_t^i-\overline{V}_0^i-\sqrt{2\sigma}\int_0^t \mathcal{R}(\overline{V}_s^i)\,d\overline{B}_s^i\in
	 \int_0^t\tilde{F}[\overline{\mu}_s^N](\overline{X}_s^i,\overline{V}_s^i) \,ds,
\]
where $\overline{\mu}_t^N := \frac 1N \sum_{j=1}^N \delta_{(\overline{X}_t^j, \overline{V}_t^j)}$. This completes the proof. 
\end{proof}
	
\section{Nonlinear stochastic integral system: Proof of Theorem \ref{thm_SDE}}\label{sec:ext_sde}
The purpose of this section is to show the existence and uniqueness of solutions to the following Nonlinear SIEs:
	\begin{equation}
		\label{SDE:Tr}
		\left\{ \begin{array}{ll}
			\displaystyle	Y_t = Y_0+\int_0^t W_s\,ds, & \\[3mm]
			\displaystyle W_t = W_0+ \int_0^t F[f_s](Y_s,W_s)\,ds+  \sqrt{2\sigma}\int_0^t\mathcal{R}(W_s)\,dB_s, & \\[3mm]
			f_t=\LL(Y_t,W_t).& 
		\end{array} \right.
	\end{equation}
For this, we first give the so called weak-strong stability estimate under the assumptions ${\bf (H1)}$-${\bf (H2)}$ for the set-valued function $K(\cdot)$.
	\begin{lemma}\label{lem:roparg}
		Let $(Y,W)$ and $(Y',W')$ two random variables on $\R^{2d}$ and denote $f=\LL(Y,W)$ and $f'=\LL(Y',W')$. Assume that $f\in L^{\infty}(\R^{2d})$ and both $f$ and $f'$ are compactly supported in velocity in $\bv$. Then there exists a constant depending only on $V_{m}$ such that
		$$
		\mathbb{E}\lt[|F[f](Y,W)-F[f'](Y',W')|\rt]\leq C\|f\|_{L^1\cap L^\infty}\mathbb{E}\lt[|Y-Y'|+|W-W'|\rt].
		$$ 
	\end{lemma}
	\begin{proof}
		Introducing $\pi:=\LL((Y,W),(Y',W'))$, we obtain
		$$\begin{aligned}
		&\mathbb{E}\lt[|F[f](Y,W)-F[f'](Y',W')|\rt]\cr
		&\quad =\mathbb{E}\lt[\lt|\int_{\R^{2d} \times \R^{2d}}\lt( \mb_{K(W)}(y-Y)(w-W)-\mb_{K(W')}(y'-Y')(w'-W')\rt)\pi(dy,dw,dy',dw')\rt|\rt]\cr
		&\quad \leq \mathbb{E}\lt[\lt|\int_{\R^{2d} \times \R^{2d}}\lt( \mb_{K(W)}(y-Y)(w-W)-\mb_{K(W)}(y-Y)(w'-W')\rt)\pi(dy,dw,dy',dw')\rt|\rt]\cr
		&\qquad +\mathbb{E}\lt[\lt|\int_{\R^{2d} \times \R^{2d}}\lt( \mb_{K(W)}(y-Y)-\mb_{K(W)}(y'-Y')\rt)(w'-W')\,\pi(dy,dw,dy',dw')\rt|\rt]\cr
		&\qquad + \mathbb{E}\lt[\lt|\int_{\R^{2d} \times \R^{2d}}\lt( \mb_{K(W)}(y'-Y')-\mb_{K(W')}(y'-Y')\rt)(w'-W')\,\pi(dy,dw,dy',dw')\rt|\rt]\cr
		&\quad =: I_1+I_2+I_3. 
		\end{aligned}$$	
		$\diamond$ Estimate $I_1$: First, we easily obtain
		$$
		I_1\leq \mathbb{E}\lt[|W-W'|\rt]+\int_{\R^{2d} \times \R^{2d}} |w-w'|\pi(dy,dw,dy',dw')=2\mathbb{E}\lt[|W-W'|\rt]. 
		$$
		$\diamond$ Estimate $I_2$:
		Using the fact that $f,f'$ are compactly supported in velocity together with Lemma \ref{lem_est2}, we find
		$$\begin{aligned}
		I_2&\leq 2V_{m}\mathbb{E}\lt[\int_{\R^{2d} \times \R^{2d}} \lt|\mb_{K(W)}(y-Y)-\mb_{K(W)}(y'-Y')\rt|\pi(dy,dw,dy',dw')\rt]\cr
		&\leq 2V_{m}\mathbb{E}\lt[\int_{\R^{2d} \times \R^{2d}} \lt(\mb_{\partial^{2|Y-Y'|}K(W)}(y-Y)+\mb_{\partial^{2|y-y'|}K(W)}(y-Y)\rt)\pi(dy,dw,dy',dw')\rt]\\
		&=: I_2^1+I_2^2, 
		\end{aligned}$$
		where $I_2^1$ can be estimated by
		$$\begin{aligned} 
		I_2^1&\leq 2V_{m}\mathbb{E}\lt[ \left \| f \right \|_{L^\infty}|\Theta(W)^{2|Y-Y'|,+}|\mb_{2|Y-Y'|\leq 1}\rt]+2V_{m}\mathbb{E} \lt[|Y-Y'|\mb_{2|Y-Y'|>1} \rt]\\
		&\leq C\| f \|_{L^1\cap L^\infty}\mathbb{E}\lt[|Y-Y'|\rt],
		\end{aligned}$$
		thanks to ${\bf(H2)}$ (i)-(ii). Then by using Fubini's theorem together with the fact $f=\mathcal{L}(Y,W)$, ${\bf(H2)}$ (i)-(ii) as the above, and compact support  of $f$ in velocity, we obtain
		$$\begin{aligned}
		I_2^2 &=2V_{m}\E\lt[\int_{\R^{2d} \times \R^{2d}} \mb_{\pa^{2|y-y'|}K(u)}(y-z)f(dz,du)\pi(dy,dw,dy',dw')\rt]\\
		&\leq C(1+\|f\|_{L^\infty})\E\lt[\int_{\R^{2d} \times \R^{2d}} |y-y'|\pi(dy,dw,dy',dw')\rt]\\
		&= C\|f\|_{L^1\cap L^\infty}\mathbb{E}\lt[|Y-Y'|\rt].
		\end{aligned}$$
		$\diamond$ Estimate of $I_3$: Using $\bf{(H2)}$ (iii) together with the fact that 
		\[
		\mb_{A}(x)\leq \mb_{A^{|x-y|,+}}(y) \quad \mbox{for any set $A\subset \R^d$ and $x,y\in \R^d$},
		\]
		we split $I_3$ into two terms:
		$$\begin{aligned}
		I_3&\leq 2V_{m}\E\lt[\int_{\R^{2d} \times \R^{2d}} \mb_{K(W)\Delta K(W')}(y'-Y')\pi(dy,dw,dy',dw')\rt] \\
		& \leq 2V_{m}\E\lt[\int_{\R^{2d} \times \R^{2d}} \mb_{\Theta(W)^{C|W-W'|,+}}(y'-Y')\pi(dy,dw,dy',dw')\rt]\\
		&\leq  2V_{m}\E\lt[\int_{\R^{2d} \times \R^{2d}} \mb_{\Theta(W)^{C|W-W'|+|Y-Y'|+|y-y'|,+}}(y-Y)\pi(dy,dw,dy',dw')\rt]\\
		&\leq  2V_{m}\E\lt[\int_{\R^{2d} \times \R^{2d}}\lt( \mb_{\Theta(W)^{2C|W-W'|+2|Y-Y'|,+}}(Y-y)+\mb_{\Theta(W)^{2|y-y'|,+}}(y-Y)\rt)\pi(dy,dw,dy',dw')\rt]\\
		&=: I_3^1+I_3^2.
		\end{aligned}$$	
		We first easily estimate $I_3^1$ as
		$$
		I_3^1\leq 2V_{m}\|f\|_{L^1\cap L^\infty} \E\lt[|W-W'|+|Y-Y'|\rt],
		$$
		due to $\bf{(H2)}$ (ii). For the estimate of $I_3^2$, we again use Fubini's theorem and compact support of $f$ in velocity, we have similarly as the estimate of $I_2^2$ above that
		$$\begin{aligned}
		I_3^2 & = C\E\lt[\int_{\R^{2d} \times \R^{2d}} \mb_{\Theta(u)^{2|y-y'|,+}}(y-z)f(dz,du)\pi(dy,dw,dy',dw')\rt]\\
		&\leq C\|f\|_{L^1\cap L^\infty}  \E\lt[\int_{\R^{2d} \times \R^{2d}}|y-y'|\pi(dy,dw,dy',dw')\rt]\\
		& =C\|f\|_{L^1\cap L^\infty}  \mathbb{E}\lt[|Y-Y'|\rt].
		\end{aligned}
		$$
		By combining all the above estimates, we conclude our desired result.
\end{proof} 
We now consider the following regularized nonlinear SIEs:	
\begin{equation}\label{SDE:regTr}
\left\{ \begin{array}{ll}
\displaystyle	Y^{\eta,\e}_t = Y_0+\int_0^t W^{\eta,\e}_s\,ds, & \\[3mm]
\displaystyle W^{\eta,\e}_t = W_0+ \int_0^t F^{\eta,\e}[f^{\eta,\e}_s](Y^{\eta,\e}_s,W^{\eta,\e}_s)\,ds+  \sqrt{2\sigma}\int_0^t \mathcal{R}(W^{\eta,\e}_s)\,dB_s, & \\[3mm]
f^{\eta,\e}_t=\LL(Y^{\eta,\e}_t,W^{\eta,\e}_t),& 
\end{array} \right.
\end{equation}
where $F^{\eta,\e}$ is defined as in \eqref{eq_f}. Due to the smoothness of the force fields and the diffusion coefficients, it is clear the global existence and uniqueness of solutions for the system \eqref{SDE:regTr}.

In the lemma below, we provide the upper bound estimate of the solution $W^{\eta,\e}_t$ to the above system whose proof can be obtained by using the almost same argument as in Lemma \ref{vel_bou_trun}. 
\begin{lemma}\label{vel_bou_trun2}
Let $T>0$, and suppose that there exists a solution on the time interval $[0,T]$ to the system \eqref{SDE:Tr} with the law of the initial data $f_0=\LL(Y_0,W_0)$ which is compactly supported in velocity in $\bv$. Then it holds 
\[
|W^{\eta,\e}_t|\leq V_{m},  \quad \mathbb{P} \mbox {- a.s.,}
\]
for $\eta,\e>0$ and $t\in [0,T]$.
\end{lemma}
In the proposition below, we show the existence of weak solutions to the corresponding Vlasov-Fokker-Planck type equation to \eqref{SDE:regTr}.
\begin{proposition}\label{prop_ext}The family of time marginals of the solution $(f^{\eta,\e}_t)$ to the system \eqref{SDE:regTr} is a global-in-time weak solution of the following kinetic equation:
\bq\label{reg_vla}
\pa_t f^{\eta,\e}_t + v \cdot \nabla_x  f^{\eta,\e}_t  + \nabla_v \cdot \lt(F^{\eta,\e}[ f^{\eta,\e}_t ] f^{\eta,\e}_t  \rt) = \sigma\Delta_v (\mathcal{R}^2(v)  f^{\eta,\e}_t), \quad (x,v) \in \R^d \times \R^d,
\eq
with the initial data $f_0 = \mathcal{L}(Y_0,W_0)\in L^\infty(\R^d \times \R^d)$ compactly supported in velocity in $\bv$, where the velocity alignment force $F^{\eta,\e}[ f^{\eta,\e}_t ]$ is given by
\[
F^{\eta,\e}[ f^{\eta,\e}_t ](x,v) = \int_{\R^d \times \R^d} \mb^{\eta,\e}_{K(v)}(x-y)(w-v) f^{\eta,\e}_t(dy,dw).
\]
Furthermore, we have
\[
\sup_{0 \leq t \leq T}\|f^{\eta,\e}_t\|_{L^\infty} \leq C.
\]
for some $T > 0$ and positive constant $C>0$ independent of $\eta,\e$.
\end{proposition}
\begin{proof} For $\phi \in \mc_c^\infty(\R^{2d})$, we apply It\^o's formula to the system \eqref{SDE:regTr} to find
$$\begin{aligned}
&\phi(Y_t^{\eta,\e},W_t^{\eta,\e})\cr
&\quad =\phi(Y_0,W_0)+\int_0^t \left \langle  \nabla_{x,v}\phi(Y_s^{\eta,\e},W_s^{\eta,\e}),d(Y_s^{\eta,\e},W_s^{\eta,\e}) \right \rangle+\frac{1}{2}\int_0^t\Delta_{x,v}\phi(Y_s^{\eta,\e},W_s^{\eta,\e})d\left \langle (Y_s^{\eta,\e},W_s^{\eta,\e}) \right \rangle\\
&\quad=\phi(Y_0,W_0)+\int_0^t W_s^{\eta,\e}\cdot\nabla_x \phi(Y_s^{\eta,\e},W_s^{\eta,\e})ds+\int_0^tF^{\eta,\e}[f_s^{\eta,\e}](Y_s^{\eta,\e},W_s^{\eta,\e})\cdot\nabla_v\phi(Y_s^{\eta,\e},W_s^{\eta,\e})ds\\
&\qquad +\sqrt{2\sigma}\int_0^t \left \langle \nabla_v \phi(Y_s^{\eta,\e},W_s^{\eta,\e}),\mathcal{R}(W^{\eta,\e}_s)dB_s \right \rangle +\sigma \int_0^t \Delta_v\phi(Y_s^{\eta,\e},W_s^{\eta,\e})\mathcal{R}^2(W^{\eta,\e}_s)\,ds.
\end{aligned}$$
We then take the expectation and use the fact that $f^{\eta,\e}_t =\mathcal{L}(Y^{\eta,\e}_t,W_t^{\eta,\e})$ to obtain
$$\begin{aligned}
&\int_{\R^d\times \R^d}\phi(x,v)f_t^{\eta,\e}(dx,dv)\cr
&=\int_{\R^d\times \R^d}\phi(x,v)f_0(dx,dv)+\int_0^t \int_{\R^d\times \R^d} v\cdot\nabla_x \phi(x,v)f^{\eta,\e}_s(dx,dv)\,ds\\
&\quad \,\,+\int_0^t\int_{\R^d\times \R^d} F^{\eta,\e}[f_s^{\eta,\e}](x,v)\cdot\nabla_v\phi(x,v)f^{\eta,\e}_s(dx,dv)\,ds\cr
&\quad \,\,+\sigma\int_0^t \int_{\R^d\times \R^d} \Delta_v\phi(x,v) \mathcal{R}^2(v)f^{\eta,\e}_s(dx,dv)\,ds.
\end{aligned}$$
This concludes that the family of time marginals of the process solutions $(Y^{\eta,\e}_t,W_t^{\eta,\e})$ to the system \eqref{SDE:regTr} is a weak solution for the kinetic equation \eqref{reg_vla}. 

We next show the uniform bound estimate of $\|f^{\eta,\e}\|_{L^\infty}$ in $\eta,\e > 0$. Due to the diffusion term, we first estimate $L^p$-norm of the solution $f^{\eta,\e}$ and send $p \to \infty$ to have the desired $L^\infty$-estimate. For $p \geq 1$, we obtain
$$\begin{aligned}
&\frac{d}{dt} \int_{\R^d\times \R^d} (f_t^{\eta,\e})^p\,dxdv	\cr
&\quad = p\int_{\R^d\times \R^d} \partial_t f^{\eta,\e}_t(f_t^{\eta,\e})^{p-1}dxdv\\
&\quad =- p\int_{\R^d\times \R^d} \nabla_v \cdot (F^{\eta,\e}[f_t^{\eta,\e}]f_t^{\eta,\e})(f_t^{\eta,\e})^{p-1}dxdv+\sigma p\int_{\R^d\times \R^d}\Delta_v (\mathcal{R}^2(v) f_t^{\eta,\e})(f_t^{\eta,\e})^{p-1}dxdv\\
&\quad =: J_1 + J_2,
\end{aligned}$$
where $J_i,i=1,2$ are estimated as follows.
$$\begin{aligned}
J_1 &= - p \int_{\R^d \times \R^d} \lt(\nabla_v \cdot F^{\eta,\e}[f_t^{\eta,\e}] \rt)(f_t^{\eta,\e})^p\,dxdv - \int_{\R^d \times \R^d} F^{\eta,\e}[f_t^{\eta,\e}] \cdot\nabla_v (f_t^{\eta,\e})^p\,dxdv\cr
&= - (p-1) \int_{\R^d \times \R^d} \lt(\nabla_v \cdot F^{\eta,\e}[f_t^{\eta,\e}] \rt)(f_t^{\eta,\e})^p\,dxdv \cr
J_2 &= -\sigma p \int_{\R^d \times \R^d} \lt(\nabla_v (\mathcal{R}^2(v))f^{\eta,\e}_t + \mathcal{R}^2(v) \nabla_v f^{\eta,\e}_t \rt) \cdot \nabla_v ((f^{\eta,\e}_t)^{p-1})\,dxdv\cr
&=\sigma (p-1)\int_{\R^d \times \R^d} \Delta_v (\mathcal{R}^2(v))(f^{\eta,\e}_t)^p\,dxdv - \sigma p(p-1)\int_{\R^d \times \R^d} \mathcal{R}^2(v) (f^{\eta,\e}_t)^{p-2}|\nabla_v f^{\eta,\e}_t|^2\,dxdv.
\end{aligned}$$
This yields
$$\begin{aligned}
\frac{d}{dt}\|f^{\eta,\e}_t\|_{L^p}^p \leq (p-1)\lt(\|\nabla_v \cdot F^{\eta,\e}[f_t^{\eta,\e}]\|_{L^\infty} + \sigma \|\Delta_v \mathcal{R}^2(v)\|_{L^\infty}\rt)\|f^{\eta,\e}_t\|_{L^p}^p.
\end{aligned}$$
On the other hand, it follows from \cite[Proposition 4.1]{CCHS} together with Lemma \ref{vel_bou_trun2} that
\[
\|\nabla_v \cdot F^{\eta,\e}[f_t^{\eta,\e}]\|_{L^\infty} \leq C(1 + \|f^{\eta,\e}_t\|_{L^\infty}),
\]
where $C > 0$ is independent of $\eta$, $\e$, and $p$. Thus we obtain
\[
\frac{d}{dt}\left\|f_t^{\eta,\e} \right\|^p_{L^p} \leq Cp(1+\|f_t^{\eta,\e}\|_{L^\infty} )\left\|f_t^{\eta,\e} \right\|^p_{L^p}.
\]
By applying Gronwall's inequality, we get
\[ 
\|f_t^{\eta,\e}\|_{L^p} \leq \|f_0\|_{L^p} \exp\lt( C\int_0^t (1 + \|f_s^{\eta,\e}\|_{L^\infty})ds\rt).
\]
We now send $p \to \infty$ to find
\[
\|f_t^{\eta,\e}\|_{L^\infty} \leq \|f_0\|_{L^\infty} \exp\lt( C\int_0^t (1 + \|f_s^{\eta,\e}\|_{L^\infty})\,ds\rt).
\]
Set 
\[
g(t):= \|f_0\|_{L^\infty} \exp\lt( C\int_0^t (1 + \|f_s^{\eta,\e}\|_{L^\infty})ds\rt) \geq 0.
\]
Then $g(0) = \|f_0\|_{L^\infty}$ and $g$ satisfies 
\[
g'(t) = Cg(t)(1 + \|f_t^{\eta,\e}\|_{L^\infty}) \leq Cg(t)(1 + g(t)).
\]
This yields
\[
(g^{-1}(t))' + Cg^{-1}(t) \geq -C \quad \mbox{and} \quad g(t) \leq \frac{1}{(1 + g_0^{-1})e^{-Ct} - 1}.
\]
Hence we have
\[
\|f_t^{\eta,\e}\|_{L^\infty} \leq g(t) \leq \frac{1}{(1 + g_0^{-1})e^{-Ct} - 1} =  \frac{\|f_0\|_{L^\infty}}{(1 + \|f_0\|_{L^\infty})e^{-Ct} - \|f_0\|_{L^\infty}}.
\]
\end{proof}
	
We then provide the existence and stability of strong solutions to the nonlinear SIEs \eqref{SDE:Tr} in the proposition below.
\begin{proposition}\label{prop:exNLSDER}There exists only one process on the time interval $[0,T]$ for some $T>0$ solving the system \eqref{SDE:Tr} in the strong sense starting from $(Y_0,W_0)$ with law $f_0\in L^{\infty}(\R^{2d})$, independent of $(B_t)_{t\in[0,T]}$ such that its time marginal satisfies $f_t \in L^\infty(0,T;L^\infty(\R^{2d}))$ and solves the kinetic equation \eqref{sys_kin_Trun}. Moreover if $(Y_t,W_t)_{t \in [0,T]}$ and $(Y'_t,W'_t)_{t \in [0,T]}$ are two solutions to the system \eqref{SDE:Tr} with the initial data $(Y_0, W_0)$ and $(Y'_0, W'_0)$, respectively, such that $f_t=\LL(Y_t,W_t)$ with $f_t \in L^\infty(0,T;L^\infty(\R^{2d}))$. Then we have
\[
\mathbb{E}\bigl[|Y_t-Y'_t|+|W_t-W'_t|\bigr]\leq e^{C\int_0^t \|f_s\|_{L^1\cap L^\infty} ds}\mathbb{E}\bigl[|Y_0-Y'_0|+|W_0-W'_0|\bigr],
\]
for $t \in [0,T]$.
\end{proposition}
\begin{proof} As mentioned in Introduction, we cannot directly use the $1$-Wasserstein distance. Set $\varphi_\gamma(x) := \sqrt{\gamma^2 + |x|^2}$ with $\gamma > 0$. Note that
\[
\nabla \varphi_\gamma(x) = \frac{x}{\sqrt{\gamma^2 + |x|^2}} \quad \mbox{and} \quad \Delta \varphi_\gamma(x) = \frac{d\gamma^2 + (d-1)|x|^2}{(\gamma^2 + |x|^2)^{3/2}}.
\]
We also find from the above that
\bq\label{est_sim}
|\nabla \varphi_\gamma(x)| \leq 1 \quad \mbox{and} \quad |\Delta \varphi_\gamma(x)| \leq \frac{d}{\sqrt{\gamma^2 + |x|^2}}.
\eq
$\diamond$ {\bf Step A (Cauchy estimates).-} Let $(Y^{\eta,\e}_t, W^{\eta,\e}_t)$ and $(Y^{\eta',\e'}_t, W^{\eta',\e'}_t)$ be solutions to the regularized system \eqref{SDE:regTr}. Then, by applying It\^o's formula, we obtain
\begin{align}\label{est_cau}
\begin{aligned}
&\sqrt{\gamma^2 + |W^{\eta,\e}_t - W^{\eta',\e'}_t|^2}\cr
&\quad = \int_0^t \lt(\frac{W^{\eta,\e}_s - W^{\eta',\e'}_s}{\sqrt{\gamma^2 + |W^{\eta,\e}_s - W^{\eta',\e'}_s|^2}} \rt)\cdot\lt(F^{\eta,\e}[f^{\eta,\e}_s](Y^{\eta,\e}_s,W^{\eta,\e}_s) - F^{\eta',\e'}[f^{\eta',\e'}_s](Y^{\eta',\e'}_s,W^{\eta',\e'}_s) \rt)ds\cr
&\qquad + \sqrt{2\sigma}\int_0^t  \left(\mathcal{R}(W^{\eta,\e}_s)-\mathcal{R}(W^{\eta',\e'}_s)\right)\lt(\frac{W^{\eta,\e}_s - W^{\eta',\e'}_s}{\sqrt{\gamma^2 + |W^{\eta,\e}_s - W^{\eta',\e'}_s|^2}} \rt) \cdot dB_s\cr
&\qquad + \sigma\int_0^t \lt(\frac{d\gamma^2 + (d-1)|W^{\eta,\e}_s - W^{\eta',\e'}_s|^2}{(\gamma^2 + |W^{\eta,\e}_s - W^{\eta',\e'}_s|^2)^{3/2}} \rt)\left(\mathcal{R}(W^{\eta,\e}_s)-\mathcal{R}(W^{\eta',\e'}_s)\right)^2\,ds,
\end{aligned}
\end{align}
due to $W^{\eta,\e}_0 = W^{\eta',\e'}_0 = W_0$. We next take the expectation to the above equality to find 
$$\begin{aligned}
&\E\lt[\sqrt{\gamma^2 + |W^{\eta,\e}_t - W^{\eta',\e'}_t|^2} \rt]\cr
&\quad \leq \int_0^t \E\lt[\lt|F^{\eta,\e}[f^{\eta,\e}_s](Y^{\eta,\e}_s,W^{\eta,\e}_s) - F^{\eta',\e'}[f^{\eta',\e'}_s](Y^{\eta',\e'}_s,W^{\eta',\e'}_s)\rt| \rt]ds\cr
&\qquad + \sigma d\int_0^t \E\lt[\frac{1}{\sqrt{\gamma^2 + |W^{\eta,\e}_s - W^{\eta',\e'}_s|^2}}\left(\mathcal{R}(W^{\eta,\e}_s)-\mathcal{R}(W^{\eta',\e'}_s)\right)^2 \rt]ds\cr
&\quad =: \int_0^t K_1 \,ds + \int_0^t K_2 \,ds,
\end{aligned}$$
where we used \eqref{est_sim}. Here $K_2$ can be easily estimated as
$$\begin{aligned}
\int_0^t K_2 \,ds &\leq \sigma d\|\mathcal{R}\|_{Lip}^2\int_0^t \E\lt[\frac{|W^{\eta,\e}_s - W^{\eta',\e'}_s|^2}{\sqrt{\gamma^2 + |W^{\eta,\e}_s - W^{\eta',\e'}_s|^2}} \rt]ds \cr
&\leq \sigma d\|\mathcal{R}\|_{Lip}^2\int_0^t \lt(\gamma + \E[|W^{\eta,\e}_s - W^{\eta',\e'}_s|]\rt)ds.
\end{aligned}$$
For the estimate of $K_1$, we decompose it as
$$\begin{aligned}
K_1 &\leq \mathbb{E}\left[\left|F^{\eta,\e}[f^{\eta,\e}_s](Y_s^{\eta,\e},W_s^{\eta,\e})-F[f^{\eta,\e}_s](Y_s^{\eta,\e},W_s^{\eta,\e})\right|\right]\cr
&\quad + \mathbb{E}\left[\left|F[f^{\eta,\e}_s](Y_s^{\eta,\e},W_s^{\eta,\e})-F[f^{\eta',\e'}_s](Y_s^{\eta',\e'},W_s^{\eta',\e'})\right|\right]\cr
&\quad + \mathbb{E}\left[\left|F[f^{\eta',\e'}_s](Y_s^{\eta',\e'},W_s^{\eta',\e'})-F^{\eta',\e'}[f^{\eta',\e'}_s](Y_s^{\eta',\e'},W_s^{\eta',\e'})\right|\right]\cr
&=: K_1^1 + K_1^2 + K_1^3.
\end{aligned}$$
$\diamond$ Estimates of $K_1^1$ and $K_1^3$: We again split $K_1^1$ into two terms:
$$\begin{aligned}
K_1^1&=\mathbb{E}\lt[\lt|\int_{\R^d \times \R^d}\lt(\mb^{\eta,\e}_{K(W^{\eta,\e}_s)}(Y_s^{\eta,\e}-y)-\mb_{K(W^{\eta,\e}_s)}(Y_s^{\eta,\e}-y)\rt)(w-W_s^{\eta,\e}) f_s^{\eta,\e}(dy,dw) \rt|\rt]\\
&\leq \mathbb{E}\lt[\int_{\R^d \times \R^d}\lt|\mb^{\eta,\e}_{K(W^{\eta,\e}_s)}(Y_s^{\eta,\e}-y)-\mb^{\e}_{K(W^{\eta,\e}_s)}(Y_s^{\eta,\e}-y)\rt||w-W_s^{\eta,\e}|f_s^{\eta,\e}(dy,dw) \rt]\\
& \quad + \mathbb{E}\lt[\int_{\R^d \times \R^d}\lt|\mb^{\e}_{K(W^{\eta,\e}_s)}(Y_s^{\eta,\e}-y)-\mb_{K(W^{\eta,\e}_s)}(Y_s^{\eta,\e}-y)\rt||w-W_s^{\eta,\e}|f_s^{\eta,\e}(dy,dw) \rt]\\
&=: K_1^{1,1}+K_1^{1,2}.
\end{aligned}$$			
Using the estimate of compact support of $f_s^{\eta,\e}$ in velocity and \eqref{lem_add},  we find for $\eta \leq 1$
$$\begin{aligned}
K_1^{1,1} &\leq 2V_m|\bv| \|f_s^{\eta,\e}\|_{L^\infty}\mathbb{E}\lt[\int_{\R^d}\lt|\mb^{\eta,\e}_{K(W^{\eta,\e}_s)}(Y_s^{\eta,\e}-y)-\mb^{\e}_{K(W^{\eta,\e}_s)}(Y_s^{\eta,\e}-y)\rt|dy\rt]\\
&\leq C\|f_s^{\eta,\e}\|_{L^\infty}\eta.
\end{aligned}$$
Similarly, by using \eqref{lem_diff1} and $\bf{(H2)}$ (i)-(ii), we get for $\e<1/2$
$$\begin{aligned}
K_1^{1,2} &\leq 2V_{m}|\bv| \|f_s^{\eta,\e}\|_{L^\infty}\mathbb{E}\lt[\int_{\R^d}\lt|\mb^{\e}_{K(W^{\eta,\e}_s)}(Y_s^{\eta,\e}-y)-\mb_{K(W^{\eta,\e}_s)}(Y_s^{\eta,\e}-y)\rt|dy\rt]\\
&\leq C\|f_s^{\eta,\e}\|_{L^\infty}\sup_{v\in \R^d}|\Theta(v)^{2\e,+}|\cr
&\leq  C\|f_s^{\eta,\e}\|_{L^\infty}\e.
\end{aligned}$$
Employing the almost same argument as above, we estimate $K_1^3$ as
\[
K_1^3 \leq C\|f_s^{\eta,\e}\|_{L^\infty}\lt(\eta' + \e' \rt) \quad \mbox{for} \quad \eta', \e' < 1/2.
\]
$\diamond$ Estimate of $K_1^2$: It follows from Lemma \ref{lem:roparg} that
$$
K_1^2 \leq C(1+\|f_s^{\eta,\e}\|_{L^\infty})\mathbb{E}\lt[|Y_s^{\eta,\e}-Y_s^{\eta',\e'}|+|W_s^{\eta,\e}-W_s^{\eta',\e'}|\rt].
$$ 
Combining the all of above estimates, we have
$$\begin{aligned}
&\mathbb{E}\lt[|Y^{\eta,\e}_t-Y^{\eta',\e'}_t|+\sqrt{\gamma^2+|W^{\eta,\e}_t-W^{\eta',\e'}_t|^2}\rt]\cr
&\quad \leq\gamma+ C\int_0^t\left ( 1+\|f_s^{\eta,\e}\|_{L^\infty} \right )\mathbb{E}\lt[|Y^{\eta,\e}_s-Y^{\eta',\e'}_s|+\sqrt{\gamma^2+|W^{\eta,\e}_s-W^{\eta',\e'}_s|^2}\rt]ds \cr
&\qquad +C\int_0^t\|f_s^{\eta,\e}\|_{L^\infty}(\eta+\e+\eta'+\e')\,ds,
\end{aligned}$$
where $C >0$ is independent of $\eta,\e,\eta', \e'$ and $\gamma$. We finally let $\gamma \to 0$ and use Gronwall's inequality together with the uniform bound estimate in Proposition \ref{prop_ext}, we have
\[
\sup_{t\in [0,T]}\mathbb{E}\lt[|Y^{\eta,\e}_t-Y^{\eta',\e'}_t|+|W^{\eta,\e}_t-W^{\eta',\e'}_t|\rt]\leq C(\eta+\e+\eta'+\e') \quad \mbox{for} \quad \eta,\e,\eta',\e' < 1/2,
\]
where the constant $C$ depends only on $V_{m},d, \|\mathcal{R}\|_{Lip}$, and $T$.

\noindent $\diamond$ {\bf Step B (Uniform-in-time estimate).-} Before taking the expectation to the equality \eqref{est_cau}, we first take the supremum over the time interval $[0,T]$. Then we can estimate similarly as in {\bf Step A} to estimate 
\[
\mathbb{E}\lt[\sup_{t\in [0,T]}\lt(|Y^{\eta,\e}_t-Y^{\eta',\e'}_t|+|W^{\eta,\e}_t-W^{\eta',\e'}_t|\rt)\rt]
\]
except the following term 
\[
\int_0^t  \left(\mathcal{R}(W^{\eta,\e}_s)-\mathcal{R}(W^{\eta',\e'}_s)\right)\lt(\frac{W^{\eta,\e}_s - W^{\eta',\e'}_s}{\sqrt{\gamma^2 + |W^{\eta,\e}_s - W^{\eta',\e'}_s|^2}} \rt) \cdot dB_s
\]
since this term will not vanish if we take the expectation after the supremum. For this, we use Doob's inequality to estimate it as
$$\begin{aligned}
&\E\lt[\sup_{\tau \leq t}\lt|\int_0^\tau  \left(\mathcal{R}(W^{\eta,\e}_s)-\mathcal{R}(W^{\eta',\e'}_s)\right)\lt(\frac{W^{\eta,\e}_s - W^{\eta',\e'}_s}{\sqrt{\gamma^2 + |W^{\eta,\e}_s - W^{\eta',\e'}_s|^2}} \rt) \cdot dB_s\rt|
 \rt]\cr
 &\quad \leq \E\lt[\sup_{\tau \leq t}\lt|\int_0^\tau  \left(\mathcal{R}(W^{\eta,\e}_s)-\mathcal{R}(W^{\eta',\e'}_s)\right)\lt(\frac{W^{\eta,\e}_s - W^{\eta',\e'}_s}{\sqrt{\gamma^2 + |W^{\eta,\e}_s - W^{\eta',\e'}_s|^2}} \rt) \cdot dB_s\rt|^2
 \rt]^{1/2}\cr
 &\quad \leq 2\|\mathcal{R}\|_{Lip}\E\lt[\int_0^t |W^{\eta,\e}_s - W^{\eta',\e'}_s |^2ds\rt]^{1/2}.
\end{aligned}$$
Then this together with the similar estimates as in {\bf Step A} for other terms gives
\[
\E\lt[\sup_{t\in[0,T]}\lt(|Y_t^{\eta,\e}-Y_t^{\eta',\e'}|+|W_t^{\eta,\e}-W_t^{\eta',\e'}|\rt)\rt]\leq C\lt(\eta+\e +\eta'+\e' \rt),
\]
for $\eta,\e,\eta',\e' < 1/2$, where $C >0$ is independent of the regularization parameters, $\eta,\e,\eta'$, and $\e'$.

\noindent $\diamond$ {\bf Step C (Passing to the limit).-} It follows from {\bf Step B} that there exists a limit process $(Y_t,W_t)_{t\in [0,T]}$ of the $(Y^{\eta,\e}_t,W^{\eta,\e}_t)_{t\in [0,T]}$ as $\eta,\e\rightarrow 0$ in $L^1(\Omega\times [0,T])$. Using the fact that
$$
\mathcal{W}_1(f_t^{\eta,\e},f_t^{\eta',\e'})\leq \E\lt[|Y_t^{\eta,\e}-Y_t^{\eta',\e'}|+|W_t^{\eta,\e}-W_t^{\eta',\e'}|\rt],
$$
we also deduce that $(f_t^{\eta,\e})_{t\in[0,T]}$ is a Cauchy sequence in $\mc([0,T]; \mathcal{P}_1(\R^d \times \R^d))$. Then, by completeness of this space, we define $f \in \mc([0,T]; \mathcal{P}_1(\R^d \times \R^d))$ by $f_t := \lim_{\eta,\e \to 0} f_t^{\eta,\e}$ for $t \in [0,T]$. We next show that the limiting process $(Y_t,W_t)_{t\in [0,T]}$ obtained before is the solution to the nonlinear SIEs \eqref{SDE:Tr} and its time marginal $f_t$ is the weak solution to the kinetic equation \eqref{sys_kin_Trun}. For this, it is enough to show that
\[
K_3^{\eta,\e}:= \E\lt[\lt|\int_0^t \lt(F^{\eta,\e}[f^{\eta,\e}_s](Y^{\eta,\e}_s,W^{\eta,\e}_s) - F[f_s](Y_s,W_s) \rt)ds \rt| \rt] \to 0
\]
and
\[
K_4^{\eta,\e}:= \E\lt[\lt|\int_0^t \lt(\mathcal{R}(W^{\eta,\e}_s) - \mathcal{R}(W_s) \rt) \cdot dB_s \rt|\rt] \to 0
\]
as $\eta,\e \to 0$. Using similar arguments for the term $K_1$ in {\bf Step A}, we easily estimate $K_3^{\eta,\e}$ as
\[
K_3^{\eta,\e} \leq C(\eta + \e) + C(1+\|f_s\|_{L^\infty})\E\lt[|Y_s^{\eta,\e}-Y_s|+|W_s^{\eta,\e}-W_s|\rt] \to 0 \quad \mbox{as} \quad \eta,\e \to 0.
\]
We next use the Lipschitz regularity of $\mathcal{R}$ together with It\^o isometry to get
$$\begin{aligned}
K_4^{\eta,\e} &\leq \E\lt[\lt|\int_0^t \lt(\mathcal{R}(W^{\eta,\e}_s) - \mathcal{R}(W_s) \rt) \cdot dB_s \rt|^2\rt]^{1/2} \cr
&= \E\lt[\int_0^t \lt(\mathcal{R}(W^{\eta,\e}_s) - \mathcal{R}(W_s) \rt)^2 ds \rt]^{1/2}\cr
&\leq \|\mathcal{R}\|_{Lip}\,\E\lt[\int_0^t |W^{\eta,\e}_s - W_s|^2\,ds \rt]^{1/2} \to 0 \quad \mbox{as} \quad \eta,\e \to 0.
\end{aligned}$$
This concludes that the liming process $(Y_t,W_t)_{t\in [0,T]}$ is the solution to the nonlinear SIEs \eqref{SDE:Tr}. For the weak solutions to the kinetic equation \eqref{sys_kin_Trun}, we take a test function $\phi \in \mc^\infty_c(\R^{2d})$ and apply It\^o's formula to that solution of the system  \eqref{SDE:Tr} as in the proof of Proposition \ref{prop_ext}. Then we can find that its time marginals $f_t$ solves that kinetic equation \eqref{sys_kin_Trun} in the distributional sense. We also easily check that solution $f \in L^\infty(\R^{2d} \times [0,T])$.
		
\noindent	$\diamond$ {\bf Step D (Stability estimate).-} If $(Y^i_t,W^i_t), i=1,2$ are two processes obtained above with the initial data $(Y^i_0,W^i_0), i=1,2$, respectively, and $\LL(Y^i_t,W^i_t) = f_t,i=1,2, \,t \in [0,T]$, then by employing almost same argument as in {\bf Step A}, we can easily find
$$\begin{aligned}
&\mathbb{E}\lt[|Y^1_t-Y^2_t|+\sqrt{\gamma^2+|W^1_t-W^2_t|^2}\rt] \cr
&\quad \leq \mathbb{E}\lt[|Y^1_0-Y^2_0|+\sqrt{\gamma^2+|W^1_0-W^2_0|^2}\rt]\cr
&\qquad + C\int_0^t\left ( 1+\|f_s^1\|_{L^\infty} \right )\mathbb{E}\lt[|Y^1_s-Y^2_s|+\sqrt{\gamma^2+|W^1_s-W^2_s|^2}\rt]ds.
\end{aligned}$$
Applying Gronwall's inequality and letting $\gamma \to 0$ conclude the desired stability estimate. 
\end{proof}

\begin{proof}[Proof of Theorem \ref{thm_SDE}] We already discussed the existence and uniqueness of strong solutions to the nonlinear SIEs \eqref{sys_NLS_T} and the existence of weak solutions to the kinetic equation \eqref{sys_kin_Trun} in Proposition \ref{prop:exNLSDER}. For the uniqueness of weak solutions to \eqref{sys_kin_Trun}, we use the fact that for any solutions to \eqref{sys_kin_Trun} can be represented as the time marginals of some solutions to \eqref{SDE:Tr}. More precisely, let $(\tilde{f}_t)_{t \ge 0}$ be a weak solution to the equation \eqref{sys_kin_Trun} with the initial data $\tilde{f}_0\in \mathcal{P}_1(\R^d \times \R^d)$ and $(f_t)_{t \ge 0}$ be another weak solution to \eqref{sys_kin_Trun} with the initial data $f_0\in (L^1 \cap L^\infty)(\R^d \times \R^d) \cap \pp_1(\R^d \times \R^d)$ such that $f\in L^{\infty}(0,T; (L^1 \cap L^\infty)(\R^d \times \R^d))\cap \mc([0,T];\pp_1(\R^d \times \R^d))$. Then we can find a probability space $( \Omega,\mathbb{P},\left (\mathcal{F}_t \right )_{t\geq 0},\mathcal{F})$, a Brownian motion $(B_t)_{t \ge 0}$ on that basis and a process $(X_t,V_t)_{t \ge 0}$ solution to \eqref{SDE:Tr}, which has the time marginal $\tilde{f}_t$ at any time $t \ge 0$. We refer to \cite{CS16} for more details on that. On that probability space, let $(Y_0,W_0)$ be a random variable on $\R^d \times \R^d$ with the law $f_0$ independent of $(B_t)_{t\geq 0}$ such that 
\[
\mathcal{W}_1(f_0,\tilde{f}_0)=\E\lt[|X_0-Y_0|+|V_0-W_0|\rt].
\]
Since the force fields has a kind of Lipschitz property for the given $f_t$, we can construct some stochastic process $(Y_t,W_t)_{t\geq 0}$ which is a solution to \eqref{SDE:Tr} with the initial condition $(Y_0,W_0)$, and same Brownian motion as before such that its time marginal is $f_t$. Then, by using the stability estimate for SIEs \eqref{SDE:Tr} in Proposition \ref{prop:exNLSDER}, we have
$$\begin{aligned}
	\mathcal{W}_1(f_t,\tilde{f}_t)&\leq \E\lt[ |Y_t-X_t|+|V_t-W_t|\rt]\cr
	&\leq  \E\lt[ |Y_0-X_0|+|V_0-W_0|\rt]e^{\int_0^t\|f_s\|_{L^1\cap L^{\infty}} ds}\cr
	&=\mathcal{W}_1(f_0,\tilde{f}_0)e^{\int_0^t\|f_s\|_{L^1\cap L^{\infty}} ds},
\end{aligned}$$
and this concludes the uniqueness of solutions to the kinetic equation \eqref{sys_kin_Trun}.
\end{proof}

	%
	%
	%
	%
	\section{Propagation of chaos: Proof of Theorem \ref{thm:PC_T}}\label{sec:mf}

We first show law of large numbers-like estimates whose proofs rely on the nice property of our communication weights in Lemma \ref{lem_est2}. 	Even though similar results observed in \cite{CS16, HS}, for the sake of completeness we give the details of it in Appendix \ref{app_a}.

\begin{lemma}\label{lem_(De)-Poiss}
Let $(Y_i,W_i)_{i=1,\cdots,N}$ be i.i.d. random variables with law $f\in \lt(\mathcal{P}_1\cap L^{\infty}\rt)(\R^{2d})$, and let $\mu_N$ be the associated empirical measure $\mu_N=\frac{1}{N}\sum_{i=1}^N\delta_{(Y_i,W_i)}$ and a pair $(Y,W)$ be independent of $(Y_i,W_i)_{i=1,\cdots,N}$. Let us also denote by $\rho$ the first marginal of $f$ and $\rho_N=\frac{1}{N}\sum_{i=1}^N\delta_{Y_i}$. Then we have
\begin{equation*}
\mathbb{E}\lt[\sup_{u\geq 0}\left|\int_{\R^d} \mb_{\Theta(W)^{u,+}}(y-Y)(\rho_N-\rho)(dy)\right|\rt] \leq \frac{C}{\sqrt{N}}
\end{equation*}
and 
$$\mathbb{E}\lt[\sup_{u\geq 0}\left|\int_{\R^d \times \R^d} \mb_{\Theta(W)^{u,+}}(y-Y)(\mu_N-f)(dy,dw)\right|\rt] \leq \frac{C}{\sqrt{N}}. $$
	\end{lemma}	
	
	\begin{lemma}
		\label{lem:LLN}
		Let $(Y_i)_{i=1,\cdots,N}$ be i.i.d. random variables with law $\rho\in \mathcal{P}(\R^d)$, and set $h$ a measurable function in $\R^d$ such that $h(Y_1-Y_2)$ is almost surely bounded from above by some constant $c_0$ and $h(0)=0$. Then we have 
		\[
		\mathbb{E}\lt[ \lt| \int_{\R^d} h(Y_1,Y_1-y)(\rho_N - \rho)(dy)\rt| \rt] \leq \frac{C}{\sqrt{N}},
		\] 
		where $\rho_N=\frac{1}{N}\sum_{i=1}^N\delta_{Y_i}$.
	\end{lemma}

We next provide a quantitative estimate between velocity fields of the stochastic integral inclusion system \eqref{eq:SIIE} and the nonlinear SIEs \eqref{SDE:Tr}.
\begin{lemma}\label{lem:coup_T}
Let $(X_i,V_i)_{i=1,\cdots,N}$ be $N$ random variables on $\R^{2d}$ and $(Y_i,W_i)_{i=1,\cdots,N}$ be $N$ i.i.d random variables on $\R^{2d}$ with law $f\in L^{\infty}(\R^{2d})$ compactly supported in velocity in $\bv$. Then there exists a constant $C$ depending only on $V_{m},d$ such that
\[
\mathbb{E}\lt[\frac{1}{N}\sum_{i=1}^N \lt|\tilde{F}[\mu_N](X_i,V_i) -F[f](Y_i,W_i) \rt|\rt]\leq C\|f\|_{L^1\cap L^\infty}\frac{1}{N}\sum_{i=1}^N\E\lt[|X_i-Y_i|+|V_i-W_i|\rt]+ \frac{C}{\sqrt{N}},
\]
where $\mu_N=\frac{1}{N}\sum_{i=1}^N\delta_{(X_i,V_i)}$ and we abused the notation $| \cdot |$ for the distance between a set and a vector.
\end{lemma}
\begin{proof} We first introduce an empirical measure $\nu_N=\frac{1}{N}\sum_{i=1}^N\delta_{(Y_i,W_i)}$. For any $z_i\in \tilde{F}[\mu^N](X_i,V_i)$ with $i \in \{1\,\cdots,N\}$, there exists some (random variables) $(\alpha_{i,j})_{j=1,\cdots,N}\in [0,1]^N$ such that
\[
z_i=\frac{1}{N}\sum_{j=1}^N\alpha_{i,j}(V_j-V_i).
\]
Then, for $i \in \{1,\cdots,N\}$, we get
$$\begin{aligned}
&\mathbb{E}\left [ \left|z - F[f](Y_i,W_i)\right| \right ]\cr
& \quad \leq \E\lt[ \frac{1}{N}\sum_{j=1}^N\lt(\alpha_{i,j}-\mb_{K(V_i)}(X_i-X_j)\rt)(V_j-V_i)\rt]\\
& \qquad +\mathbb{E}\left [ \left|F[\nu_N](Y_i,W_i)-F[\mu_N](X_i,V_i)\right| \right ]+\mathbb{E}\left [ \left|F[\nu_N](Y_i,W_i)-F[f](Y_i,W_i)\right| \right ]\\
&\quad =: I^i_1 + I^i_2 + I^i_3.\cr
\end{aligned}$$	
$\diamond$ Estimate of $I^i_1$: Note that $\alpha_{i,j}-\mb_{K(V_i)}(X_j-X_i)\neq 0$ only if $X_j-X_i\in \tilde{\pa}K(V_i)$. This together with the assumption $\bf{(H2)}$ (iii) yields
$$\begin{aligned} |\alpha_{i,j}-\mb_{K(V_i)}(X_j-X_i)|&\leq  C\mb_{\lt\{X_j-X_i\in \tilde{\pa}K(V_i)\rt\}}\leq C\mb_{\lt\{Y_j-Y_i\in \Theta(W_i)^{C|Y_j-X_j|+|Y_i-X_i|+|V_i-W_i|,+}\rt\}}\\
&\leq C\mb_{\lt\{Y_j-Y_i\in \Theta(W_i)^{C\lt(|Y_i-X_i|+|V_i-W_i|\rt),+}\rt\}}+C\mb_{\lt\{Y_i-Y_j\in \Theta(W_i)^{C|Y_j-X_j|,+}\rt\}}.
\end{aligned}$$
Thus we obtain
$$\begin{aligned}
I_1^i&\leq CV_{m}\E\lt[ \frac{1}{N}\sum_{j=1}^N \mb_{\Theta(W_i)^{|Y_i-X_i|+|V_i-W_i|,+}}(Y_i-Y_j)+\mb_{Y_i-Y_j\in \Theta(W_i)^{C|Y_j-X_j|,+}}\rt]\\
& \leq CV_{m}\E\lt[\int_{\R^d} \mb_{\lt\{y-Y_i\in \Theta(W_i)^{C\lt(|Y_i-X_i|+|V_i-W_i|\rt),+}\rt\}}\rho(dy)\rt]\cr
&\quad + CV_{m}\E\lt[ \frac{1}{N}\sum_{j=1}^N \mb_{\lt\{Y_i-Y_j\in \Theta(W_i)^{C|Y_j-X_j|,+}\rt\}}\rt] \\
&\quad +CV_{m}\E\lt[\sup_{u\geq 0}\lt|\int_{\R^d} \mb_{\lt\{y-Y_i\in \Theta(W_i)^{u,+}\rt\}}(\rho-\rho^N)(dy)\rt|\rt].
\end{aligned}$$
Summing $I_1^i$ over $i=1,\cdots,N$ and dividing it by $N$ lead
$$\begin{aligned}
\E\lt[\frac{1}{N}\sum_{i=1}^N I_1^i\rt]&\leq C\|\rho\|_{L^1\cap L^{\infty}}\frac{1}{N}\sum_{i=1}^N\lt(|X_i-Y_i|+|V_i-W_i|\rt) \cr
&\quad +C\E\lt[\sup_{u\geq 0}\lt|\int_{\R^d} \mb_{\lt\{y-Y_i\in \Theta(W_i)^{u,+}\rt\}}(\rho-\rho^N)(dy)\rt|\rt]\cr
&\quad +C\E\lt[\sup_{u\geq 0}\lt|\int_{\R^d \times \R^d} \mb_{\lt\{y-Y_i\in \Theta(w)^{u,+}\rt\}}(f-\nu^N)(dy,dw)\rt|\rt].
\end{aligned}$$
$\diamond$ Estimate of $I^i_2$: We divide it into three terms:
$$\begin{aligned}
I_2^i &=\mathbb{E}\left [ \left| \frac{1}{N}\sum_{j=1}^N\Bigl(\mb_{K(V_i)}(X_j-X_i)(V_i-V_j)-\mb_{K(W_i)}(Y_j-Y_i)(W_i-W_j)\Bigr)\right| \right ]\cr
&=\mathbb{E}\left [ \left| \frac{1}{N}\sum_{j=1}^N\mb_{K(V_i)}(X_j-X_i)\Bigl((V_i-V_j)-(W_i-W_j)\Bigr)\right| \right ]\cr
&\quad +\mathbb{E}\left [ \left| \frac{1}{N}\sum_{j=1}^N\Bigl(\mb_{K(V_i)}(X_j-X_i)-\mb_{K(W_i)}(X_j-X_i)\Bigr)(W_i-W_j)\right| \right ]\cr
&\quad + \mathbb{E}\left [ \left| \frac{1}{N}\sum_{j=1}^N\Bigl(\mb_{K(W_i)}(X_j-X_i)-\mb_{K(W_i)}(Y_j-Y_i)\Bigr)(W_i-W_j)\right| \right ]\cr
		&=:I_{2,1}^i + I_{2,2}^i + I_{2,3}^i.
\end{aligned}$$
$\bullet$ Estimate of $I^i_{2,1}$:
We easily find 
		\[
		I_{2,1}^i\leq \E[|V_i-W_i|]+\frac{1}{N}\sum_{j=1}^N\E[|V_j-W_j|] \quad \mbox{for} \quad i=1,\cdots,N.
		\]
$\bullet$ Estimate of $I^i_{2,2}$: Using the assumptions $\bf{(H2)}$ (ii)-(iii) and the fact that $\mb_{A}(x)\leq \mb_{A^{|x-y|,+}}(y)$ for $A\subset \R^d$ and $x,y\in \R^d$, we get 
$$\begin{aligned}
&\lt|\mb_{K(V_i)}(X_j-X_i)-\mb_{K(W_i)}(X_j-X_i)\rt| \cr
&\qquad \leq \mb_{K(V_i)\Delta K(W_i)}(X_j-X_i)\cr
&\qquad \leq \mb_{\Theta(W_i)^{C|V_i-W_i|,+}}(X_j-X_i)\cr
&\qquad \leq \mb_{\Theta(W_i)^{C|V_i-W_i|+|X_i-Y_i|+|X_j-Y_j|,+}}(Y_j-Y_i)\cr
&\qquad \leq \mb_{\Theta(W_i)^{2C|V_i-W_i|+2|X_i-Y_i|,+}}(Y_j-Y_i)+\mb_{\Theta(W_i)^{2|X_j-Y_j|,+}}(Y_j-Y_i).
\end{aligned}$$
Regarding the first term on the right hand side of the above last inequality, we obtain 
$$\begin{aligned}
&\frac{1}{N}\sum_{j=1}^N \mb_{\Theta(W_i)^{2C|V_i-W_i|+2|X_i-Y_i|,+}}(Y_j-Y_i)\cr
&\quad =\int_{\R^d}\mb_{\Theta(W_i)^{2C|V_i-W_i|+2|X_i-Y_i|,+}}(y-Y_i)\,\rho_N(dy)\\
&\quad\leq \int_{\R^d}\mb_{\Theta(W_i)^{2C|V_i-W_i|+2|X_i-Y_i|,+}}(y-Y_i)\,\rho(dy)\cr
&\qquad +\lt|\int_{\R^d}\mb_{\Theta(W_i)^{2C|V_i-W_i|+2|X_i-Y_i|,+}}(y-Y_i)\,(\rho_N - \rho)(dy)\rt|\\
&\quad \leq C(\|\rho\|_{L^1 \cap L^\infty})\bigl(|V_i-W_i|+|X_i-Y_i|\bigr)+\sup_{u\geq 0}\lt|\int_{\R^d}\mb_{\Theta(W_i)^{u,+}}(y-Y_i)\,(\rho_N - \rho)(dy)\rt|,
\end{aligned}$$
for each $i = 1,\cdots, N$, where we used the assumption $\bf{(H2)}$ (ii). Similarly, for each $j=1,\cdots,N$, we find
		$$\begin{aligned}
		&\frac{1}{N}\sum_{i=1}^N\mb_{\Theta(W_i)^{2|X_j-Y_j|,+}}(Y_j-Y_i)\cr
		&\quad =\int_{\R^d} \mb_{\Theta(w)^{2|X_j-Y_j|,+}}(Y_j-y)\,\nu_N(dy)\\
		& \quad \leq \int_{\R^d \times \R^d} \mb_{\Theta(w)^{2|X_j-Y_j|,+}}(Y_j-y)\,f(dy,dw) + \lt|\int_{\R^d \times \R^d} \mb_{\Theta(w)^{2|X_j-Y_j|,+}}(Y_j-y)(\nu_N - f)(dy,dw)\rt|\\
		&\quad \leq C\|f\|_{L^1\cap L^\infty}|X_j-Y_j| + \sup_{u\geq 0}\lt|\int_{\R^d \times \R^d} \mb_{\Theta(w)^{u,+}}(Y_j-y)(\nu_N - f)(dy,dw)\rt|.
		\end{aligned}$$
		Combining the above estimates together with Lemma \ref{lem_(De)-Poiss}, we have
		\[
		\frac{1}{N}\sum_{i=1}^N \E[I_{2,2}^i] \leq C\|f\|_{L^1\cap L^\infty} \frac{1}{N}\sum_{i=1}^N\E\lt[|X_i-Y_i|+|V_i-W_i|\rt]+ \frac{C}{\sqrt{N}},
		\]
		due to $\|\rho\|_{L^\infty} \leq C\|f\|_{L^\infty}$ for some $C > 0$ depending on $V_m$ and $d$. \newline
		
\noindent $\bullet$ Estimate of $I^i_{2,3}$: It follows from Lemma \ref{lem_est2} that
\[
\left | \mb_{K(W_i)}(X_j-X_i)-\mb_{K(W_i)}(Y_j-Y_i) \right |\leq \mb_{\partial^{|X_i-Y_i|}K(W_i)}(Y_j-Y_i)+\mb_{\partial^{|X_j-Y_j|}K(W_i)}(Y_j-Y_i).
\]
We then define an empirical measure $\rho_{N}:=\frac{1}{N}\sum_{j=1}^N\delta_{Y_i}$ to estimate that for each $i=1,\cdots,N$
$$\begin{aligned}
&\frac{1}{N}\sum_{j=1 }^N\mb_{\partial^{2|X_i-Y_i|}K(W_i)}(Y_j-Y_i)\cr
&\quad = \int_{\R^d} \mb_{\partial^{2|X_i-Y_i|}K(W_i)}(y-Y_i)\,\rho_{N}(dy)\\
& \quad \leq \int_{\R^d} \mb_{\partial^{2|X_i-Y_i|}K(W_i)}(y-Y_i)\,\rho(dy)+\left |\int_{\R^d} \mb_{\partial^{2|X_i-Y_i|}K(W_i)}(y-Y_i)\,(\rho_{N}-\rho)(dy)  \right |\\
& \quad \leq C \left \| \rho \right \|_{L^1\cap L^\infty} |X_i-Y_i|+\sup_{u\geq 0}\left |\int_{\R^d} \mb_{\partial^{u}K(W_i)}(y-Y_i)\,(\rho_{N}-\rho)(dy)  \right |.
\end{aligned}$$
Similarly, we get
$$\begin{aligned}
&\frac{1}{N}\sum_{i=1}^N\mb_{\partial^{2|Y_j-X_j|}K(W_i)}(Y_j-Y_i)\cr
&\quad =\int_{\R^d \times \R^d} \mb_{\partial^{2|Y_j-X_j|}K(w)}(Y_j-y)\,\nu_{N}(dy,dw)\\
&\quad \leq  \int_{\R^d \times \R^d} \mb_{\partial^{2|X_j-Y_j|}K(w)}(Y_i-y)\,f(dy,dw)+ \left |\int_{\R^d \times \R^d} \mb_{\partial^{2|X_j-Y_j|}K(w)}(Y_j-y)\,(\nu_{N}-f)(dy,dw)  \right |\\
&\quad \leq C\left \| f \right \|_{L^1\cap L^\infty}|X_j-Y_j|+\sup_{u\geq 0}\left |\int_{\R^d \times \R^d} \mb_{\partial^{u}K(w)}(Y_j-y)\,(\nu_{N}-f)(dy,dw)  \right |,
\end{aligned}$$
for each $j=1,\cdots,N$. From the above estimates together with Lemma \ref{lem_(De)-Poiss}, we obtain
\[
\frac{1}{N}\sum_{i=1}^N \E[I_{2,3}^i] \leq C\|f\|_{L^1\cap L^\infty}\frac{1}{N}\sum_{i=1}^N\E\lt[|X_i-Y_i|\rt]+ \frac{C}{\sqrt{N}}.
\]
Collecting the all of above estimates, we have
\[
\frac{1}{N}\sum_{i=1}^N \E[I_2^i] \leq C\|f\|_{L^1\cap L^\infty} \frac{1}{N}\sum_{i=1}^N\E\lt[|X_i-Y_i|+|V_i-W_i|\rt]+ \frac{C}{\sqrt{N}},
\]
where $C > 0$ is independent of $N$.

\noindent $\diamond$ Estimate of $I^i_3$: For each $i=1,\cdots,N$, we define $h_i(y,w)=-\mb_{K(W_i)}(y)w$ and rewrite $I_3^i$ as
\[
I_3^i=\mathbb{E}\left [ \left|\int_{\R^{2d}} h_i(Y_i-y,W_i-w)(\nu_N - f)(dy,dw)\right| \right ].
\]
Then since $|h_i(Y_i-Y_j)(W_j-W_i)|\leq 2V_{m}$ for any $j\neq i$, it follows from Lemma \ref{lem:LLN} that
\[
\frac1N \sum_{i=1}^N\E[I_3^i] \leq \frac{C}{\sqrt{N}}.
\]
Finally, we combine the above estimates to conclude our desired result.
\end{proof}
\begin{proof}[Proof of Theorem \ref{thm:PC_T}] Define $\mu_t^N,\nu_t^N$ by empirical measures associated to systems \eqref{eq:SIIE} and \eqref{sys_NLS_T}:
\[
\mu_t^N=\frac{1}{N}\sum_{i=1}^N\delta_{(X_t^i,V_t^i)} \quad \mbox{and} \quad \nu_t^N=\frac{1}{N}\sum_{i=1}^N\delta_{(Y_t^i,W_t^i)}.
\]
Then, for $i=1,\cdots,N$ and $t \in [0,T]$, by using It\^o's formula, we find
	$$\begin{aligned}
	\sqrt{\gamma^2+|V_t^i-W_t^i|^2}&=\gamma+\int_0^t \left \langle \frac{V_s^i-W_s^i}{\sqrt{\gamma^2+|V_s^i-W_s^i|^2}},\tilde{F}[\mu_s^N](X_s^i,V_s^i)-F[f_s](Y_s^i,W_s^i) \right \rangle ds\\
	&\quad +\sqrt{2\sigma}\int_0^t \frac{\mathcal{R}(V_s^i)-\mathcal{R}(W_s^i)}{\sqrt{\gamma^2+|V_s^i-W_s^i|^2}}\left \langle V_s^i-W_s^i, \,dB^i_s \right \rangle\\
	&\quad + \sigma\int_0^t\frac{d\gamma^2 + (d-1)|V_s^i-W_s^i|^2}{\sqrt{\gamma^2+|V_s^i-W_s^i|^2}^3} \lt(\mathcal{R}(V_s^i)-\mathcal{R}(W_s^i)\rt)^2ds,
	\end{aligned}$$
due to $V^j_0 = W^j_0$ for all $j \in \{1,\cdots,N\}$. We then take the averaged value of $\sqrt{\gamma^2+|V_t^i-W_t^i|^2}$ and the expectation to that to obtain
$$\begin{aligned}
\E\lt[ \frac{1}{N}\sum_{i=1}^N\sqrt{\gamma^2+|V_t^i-W_t^i|^2}\rt]&\leq C\gamma +\int_0^t \E\lt[\lt| \frac{1}{N}\sum_{i=1}^N  \tilde{F}[\mu_s^N](X_s^i,V_s^i)-F[f_s](Y_s^i,W_s^i)\rt|\rt]ds\cr
&\quad + C\int_0^t  \E\lt[ \frac{1}{N}\sum_{i=1}^N|V_s^i-W_s^i|\rt] ds,
\end{aligned}$$
due to the Lipschitz regularity of $\mathcal{R}$. Using Lemma \ref{lem:coup_T} yields
\[
\frac{1}{N}\sum_{i=1}^N\mathbb{E}\left[\sqrt{\gamma^2+|V_t^i-W_t^i|^2}\right]\leq C\gamma + C\|f_s\|_{L^1\cap L^\infty}\frac{1}{N}\sum_{i=1}^N \int_0^t \E \lt[\sqrt{\gamma^2+|V_s^i-W_s^i|^2}\rt]ds+\frac{C}{\sqrt{N}},
\]
where $C$ is a positive constant depending only on $\mathcal{R}, T$, and $V_{m}$. Moreover, by letting $\gamma \to 0$ and using the fact 
\[
\E\lt[ |X_t^i-Y_t^i|  \rt]\leq \int_0^t\E\lt[ |V_s^i-W_s^i|\rt]ds,
\]
we obtain
$$\begin{aligned}
&\E\lt[ \frac{1}{N}\sum_{i=1}^N\lt(|X_t^i-Y_t^i| + |V_t^i-W_t^i|\rt)\rt]\cr
&\quad \leq C\int_0^t\|f_s\|_{L^1\cap L^\infty} \E\lt[\frac{1}{N}\sum_{i=1}^N(|X_s^i-Y_s^i|+|V_s^i-W_s^i|)\rt]ds+\frac{C}{\sqrt{N}}.
\end{aligned}$$
	Applying Gronwall's inequality to the above inequality gives
	\[
	\E\lt[\mathcal{W}_1(\mu_t^N,\nu_t^N)\rt]\leq \frac{C}{\sqrt{N}}\exp\lt(\int_0^t\|f_s\|_{L^1\cap L^\infty}\,ds\rt).
	\] 
	Finally, we use the convergence estimate stated in Proposition \ref{prop_fg} together with the moment estimate in Remark \ref{rmk:mom} to find that for all $t\in[0,T]$
\[
\mathbb{E}\lt[\mathcal{W}_1(\mu_t^N,f_t)\rt] \leq C \left\{ \begin{array}{ll}
N^{-1/2} + N^{-(q-1)/q} & \textrm{if $2 > d$ and $q \neq 2$}, \\[2mm]
N^{-1/2}\log(1+N) + N^{-(q-1)/q} & \textrm{if $2 = d$ and $q \neq 2$},\\[2mm]
N^{-1/d} + N^{-(q-1)/q} & \textrm{if $2 < d$ and $q \neq d/(d-1)$}.
\end{array} \right.
\]
This completes the proof.
\end{proof}

	%
	%
	%
	%

\appendix
	
\section{Law of large numbers-like estimates: Proofs of Lemmas \ref{lem_(De)-Poiss} and \ref{lem:LLN}}\label{app_a}

\begin{proof}[Proof of Lemma \ref{lem_(De)-Poiss}] Let $(Y_n,W_n)_{n\in\mathbb{N}}$ be a sequence of independent random variables with law $f\in \mathcal{P}(\mathbb{R}^{2d})$ and independent of the $(Y,W)$ and $K_N$ be a Poisson random variable of parameter $N$ independent of $(Y_n,W_n)_{n\in\mathbb{N}}$. Define  $\varrho_N $ the following random measure 
		\begin{equation*}
			\varrho_N = \sum_{i=1}^{K_N}\delta_{Y_i}, 
		\end{equation*}
		where $\delta_x(A)$ is a degenerate measure located in $x$, i.e., $\delta_x(A) = \mb_A(x)$. Then $\varrho_N $ is a Poisson random measure of intensity $N\rho$. Then it is straightforward to get $\left\|\varrho_N -N\rho_N\right\|_{TV} \leq |K_N-N|$, where $\| \cdot \|_{TV}$ represents the total variation distance of probability measures. Then we get 
\begin{equation*}
\begin{split}
&\sup_{u\geq 0}\left|\int_{\R^d} \mb_{\Theta(\upsilon)^{u,+}}(y-a)(\rho_N-\rho)(dy)\right|\cr
&\leq \frac{1}{N}\sup_{u\geq 0}\left|\int_{\R^d} \mb_{\Theta(\upsilon)^{u,+}}(y-a)(\varrho_N -N\rho)(dy)\right|  +\frac{1}{N}\sup_{u\geq 0}\left|\int_{\R^d} \mb_{\Theta(\upsilon)^{u,+}}(y-a)(\varrho_N -N\rho_N)(dy)\right|\\
				&\leq \frac{1}{N}\sup_{u\geq 0}\left|\int_{\R^d} \mb_{\Theta(\upsilon)^{u,+}}(y-a)\overline{\varrho}_N(dy)\right|+\frac{1}{N}\left\|\varrho_N -N\rho_N\right\|_{TV}\\
				&\leq  \frac{1}{N}\sup_{u\geq 0} \left | \mathcal{M}^{N}_u  \right |+\frac{|K_N-N|}{N}, \quad \mathbb{P}-a.s.,
			\end{split}
		\end{equation*}
		for all $(a,\upsilon)$, where
		\begin{equation*}
			\mathcal{M}^{N}_u=\int_{\R^d} \mb_{\Theta(\upsilon)^{u,+}}(y-a)\overline{\varrho}_N(dy) \quad \mbox{and} \quad \overline{\varrho}_N = \varrho_N  - N\rho.
		\end{equation*}
		Note that $(\mathcal{M}^{N}_u)_{u\geq 0}$ is a martingale with respect to the filtration $(\mathcal{F}_{u})_{u\geq 0}$ defined by:
		\begin{equation*}
			\mathcal{F}_{u}=\sigma\left\{\int_{\R^d} h(y) \varrho_N(dy) \ | \ supp \ h \in \Theta(\upsilon)^{u,+}+a \right\},
		\end{equation*}
		since $\varrho_N$ is the Poisson random measure. We then use Doob's inequality to obtain
		\begin{equation*}
			\mathbb{E}\lt[\sup_{u\geq 0}\left|\mathcal{M}^{N}_u\right|\rt]\leq \left( \mathbb{E}\lt[\sup_{u\geq 0}\left|\mathcal{M}^{N}_u\right|^2\rt]\right)^{1/2}\leq \sqrt{2}\left( \mathbb{E}\lt[\left|\mathcal{M}^{N}_{\infty}\right|^2\rt]\right)^{1/2}=  \sqrt{2}N^{1/2},
		\end{equation*}
		where we also used the fact that $\mathcal{M}^{N,a}_{\infty}=M_N(\R^d)-N$ and $M_N(\R^d)$ is a Poisson random variable of parameter $N\rho(\mathbb{R}^d)=N$ so that $\mathbb{V}(\mathcal{M}^{N,a}_{\infty})=N$. Moreover, by the property of Poisson random variable, we get
		\[
		\mathbb{E}\lt[|K_N-N|\rt]\leq (\mathbb{V}(K_N))^{1/2}= N^{1/2}.
		\]
		Combining the above estimates, we find
		\begin{equation}\label{EstFixRan}
			\mathbb{E}\lt[\sup_{u\geq 0}\left|\int_{\R^d} \mb_{\Theta(\upsilon)^{u,+}}(y-a)(\rho_N-\rho)(dy)\right|\rt]\leq \frac{3}{\sqrt{N}} \quad \mbox{for all} \quad a \in \R^d.
		\end{equation}
		We then use the classical property of conditional expectation to get
		\[
		\mathbb{E}\lt[\sup_{u\geq 0}\left|\int_{\R^d} \mb_{\Theta(\upsilon)^{u,+}}(y-Y)(\rho_N-\rho)(dy)\right|\rt]=	\mathbb{E}\lt[\mathbb{E}\lt[\sup_{u\geq 0}\left|\int_{\R^d} \mb_{\Theta(\upsilon)^{u,+}}(y-Y)(\rho_N-\rho)(dy)\right|\ | \ (Y,W) \rt]\rt].
		\]
		This and together with \eqref{EstFixRan} and the fact that $(Y,W)$ is independent of the $(Y_i,W_i)_{i=1,\cdots, N}$ yields
		\[
		\mathbb{E}\lt[\sup_{u\geq 0}\lt|\int_{\R^d} \mb_{\Theta(\upsilon)^{u,+}}(y-Y)(\rho_N-\rho)(dy)\rt|\ | \ (Y,W) \rt]\leq \frac{3}{\sqrt{N}}.
		\]
		Similarly, we define 
		$$ 
		M_N=\sum_{i=1}^{K_N} \delta_{(Y_i,W_i)}, 
		$$
		then we obtain
		$$\begin{aligned}
		&\sup_{u\geq 0}\left|\int_{\R^d \times \R^d} \mb_{\Theta(w)^{u,+}}(y-a)(\mu_N-f)(dy,dw)\right|\cr
		&\quad  \leq \frac{1}{N}\sup_{u\geq 0}\left|\int_{\R^d \times \R^d} \mb_{\Theta(w)^{u,+}}(y-a)(M_N-Nf)(dy,dw)\right|\cr
		&\qquad +\frac{1}{N}\sup_{u\geq 0}\left|\int_{\R^d \times \R^d} \mb_{\Theta(w)^{u,+}}(y-a)(M_N-N\mu_N)(dy,dw)\right|\\
		&\quad \leq \frac{1}{N}\sup_{u\geq 0}\left|\int_{\R^d \times \R^d} \mb_{\Theta(w)^{u,+}}(y-a)\overline{M}_N(dy,dw)\right|+\frac{1}{N}\left\|M_N-N\mu_N\right\|_{TV}\\
		&\quad \leq  \frac{1}{N}\sup_{u\geq 0} \left | \widetilde{\mathcal{M}}^{N}_u  \right |+\frac{|K_N-N|}{N},
		\end{aligned}$$
		with 
		\[
		\widetilde{\mathcal{M}}^{N}_u :=\int_{\R^d \times \R^d} \mb_{\Theta(w)^{u,+}}(y-a)\overline{M}_N(dy,dw)= \overline{M}_N (A_{a,u}), \quad A_{a,u}:=\{(y,w)\in \R^{2d} \ | \ y\in \Theta(w)^{u,+}+a \}.
		\]
Next we introduce a filtration $\widetilde{\mathcal{F}}_u=\sigma\{ r\leq u, \  M_N(A_{a,r}) \}$ and use the similar arguments as the above to deduce that $(\widetilde{\mathcal{M}}^{N}_u)_{u\geq 0}$ is a martingale. Then the rest of proof is almost same as before. 

\end{proof}

\begin{proof}[Proof of Lemma \ref{lem:LLN}] Note that the random variables $Y_2,\cdots,Y_N$ conditioned to $Y_1$ are still i.i.d. Then, by standard property on the conditional expectation, we estimate
		$$\begin{aligned}
		\mathbb{E}\lt[ \lt| \int_{\R^d} h(Y_1-y)(\rho_N - \rho)(dy)\rt|\rt] &= \mathbb{E}\lt[\mathbb{E}\lt[ \lt| \int_{\R^d} h(Y_1-y)(\rho_N - \rho)(dy)\rt|\, | \, Y_1 \rt]  \rt] \\
		&\leq \mathbb{E}\lt[ \left(\mathbb{E}\lt[ \frac{1}{N^2}\sum_{i=2}^N  |h(Y_1-Y_i)-h*\rho(Y_1)|^2  \, | \, Y_1 \rt]\right)^{1/2}\rt]\\
		&\leq \mathbb{E}\lt[ \sqrt{\frac{N-1}{N^2}}\left(\mathbb{E}\lt[ |h(Y_1-Y_2)-h*\rho(Y_1)|^2  \, | \, Y_1 \rt]\right)^{1/2} \rt].
		\end{aligned}$$	
		On the other hand, since $X_2$ is independent of $X_1$, we get
		\[
		\mathbb{E}\bigl[ |h(Y_1-Y_2)-h*\rho(Y_1)|^2  \, | \, Y_1 \bigr]=\int_{\R^d} |h(Y_1-y)-h*\rho(Y_1)|^2\rho(dy)\leq 4 c_0^2.
		\] 
		Here we also used the facts that $h(Y_1-Y_2)$ is almost surely bounded from above by $c_0$ and $Y_1$ is independent of $Y_2$ to obtain $|h*\rho(Y_1)|=\mathbb{E}\lt[ |h(Y_1-Y_2)| \, |\, Y_1 \rt]\leq c_0$.
\end{proof}

	%
	%
	%
	%
	\section*{Acknowledgments}
	\small{The authors warmly thank Professor Maxime Hauray for helpful discussion and valuable comments. YPC was supported by INHA UNIVERSITY Research Grant. (INHA-55447)}
	
	%
	%
	%
	%

\end{document}